\documentclass[12pt, reqno]{amsart}
\usepackage{amssymb,amsthm,amsfonts,amsmath, amscd}

\usepackage{hyperref}
\usepackage{mathrsfs}

\newcommand\Y{\mathbb Y}
\newcommand\Z{\mathbb Z}
\newcommand\C{\mathbb C}
\newcommand\R{\mathbb R}
\newcommand\T{\mathbb T}
\newcommand\GT{{\mathbb{GT}}}

\newcommand\LL{\mathbb L}

\newcommand\al{\alpha}
\newcommand\be{\beta}
\newcommand\ga{\gamma}
\newcommand\Ga{\Gamma}
\newcommand\de{\delta}

\newcommand\ka{\varkappa}
\newcommand\La{\Lambda}
\newcommand\la{\lambda}
\newcommand\si{\sigma}
\newcommand\epsi{\varepsilon}
\newcommand\om{\omega}
\newcommand\Om{\Omega}

\newcommand\wt{\widetilde}
\newcommand\wh{\widehat}

\newcommand\const{\operatorname{const}}
\newcommand\Dim{\operatorname{Dim}}

\newcommand\sym{{\operatorname{sym}}}
\newcommand\reg{{\operatorname{reg}}}
\newcommand\Res{{\operatorname{Res}}}

\newcommand\Sf{\mathfrak S}

\newcommand\down{{\downarrow}}
\newcommand\pd{\partial}

\newtheorem{theorem}{Theorem}[section]
\newtheorem{proposition}[theorem] {Proposition}

\newtheorem{corollary}[theorem]{Corollary}
\newtheorem{lemma}[theorem]{ Lemma}

\theoremstyle{definition}
\newtheorem{definition}[theorem]{Definition}
\newtheorem{remark}[theorem]{Remark}

\numberwithin{equation}{section}

\begin{document}

\title[The boundary of the Gelfand--Tsetlin graph]
{The boundary of the Gelfand--Tsetlin graph:\\ A new approach}

\author{Alexei Borodin}
\address{California Institute of Technology; Massachusetts Institute of Technology;
Institute for Information Transmission Problems, Russian Academy of Sciences}
\email{borodin@caltech.edu}

\author{Grigori Olshanski}
\address{Institute for Information Transmission Problems, Bolshoy
Karetny 19,  Moscow 127994, Russia; Independent University of Moscow, Russia}
\email{olsh2007@gmail.com}

\begin{abstract}
The Gelfand--Tsetlin graph is an infinite graded graph that encodes branching
of irreducible characters of the unitary groups. The boundary of the
Gelfand--Tsetlin graph has at least three incarnations ---  as a discrete
potential theory boundary, as the set of finite indecomposable characters of
the infinite-dimensional unitary group, and as the set of doubly infinite
totally positive sequences. An old deep result due to Albert Edrei and Dan
Voiculescu provides an explicit description of the boundary; it can be realized
as a region in an infinite-dimensional coordinate space.

The paper contains a novel approach to the Edrei--Voiculescu theorem. It is
based on a new explicit formula for the number of semi-standard Young tableaux
of a given skew shape (or of Gelfand--Tsetlin schemes of trapezoidal shape).
The formula is obtained via the theory of symmetric functions, and new
Schur-like symmetric functions play a key role in the derivation.
\end{abstract}

\maketitle

\tableofcontents

\section{Introduction}

\subsection{Finite characters of $S(\infty)$ and $U(\infty)$. A brief survey}

The symmetric group $S(n)$ and the unitary group $U(N)$ are two model
examples of finite and compact groups, respectively. Their irreducible characters
are basic objects of repre\-sen\-tation theory that have numerous applications.

In two remarkable papers by Thoma \cite{Tho64} and Voiculescu \cite{Vo76}
written inde\-pen\-dently and published twelve years apart from each other, the
authors discovered that the theory of characters can be nontrivially
generalized to groups $S(\infty)$ and $U(\infty)$ defined as inductive limits
of the group chains
$$
S(1)\subset S(2) \subset \dots  \quad\text{and}\quad     U(1)\subset U(2)\subset\dots\,.
$$

The original idea of Thoma (for $S(\infty)$) and Voiculescu (for $U(\infty)$)
consisted in replacing irreducible representations by factor representations
(in the sense of von Neumann) with finite trace. Then characters are still
ordinary functions on the group, and it turns out that for $S(\infty)$ and
$U(\infty)$ they depend on countably many continuous parameters. This fact
supports the intuitive feeling that these groups are ``big''.

It was later discovered (Vershik and Kerov \cite{VK81}, \cite{VK82}; Boyer
\cite{Boy83}) that the classification of finite characters of $S(\infty)$ and
$U(\infty)$ was obtained in a hidden form in earlier works of the beginning of
1950's (Aissen, Edrei, Schoenberg, and Whitney \cite{AESW51}; Aissen,
Schoenberg, and Whitney \cite{ASW52}; Edrei \cite{Ed52}, \cite{Ed53}). Those
papers solved the problems of classification of totally positive sequences
posed by Schoenberg in the end of 1940's (\cite{Sch48}). \footnote{Nowadays,
largely due to the works of Lusztig and Fomin-Zelevinsky, total positivity is a
popular subject. In 1960-70's the situation was different, and Thoma and
Voiculescu apparently were unaware of the work of Schoenberg and his
followers.}

On the other hand, Vershik and Kerov \cite{VK81}, \cite{VK82}, \cite{VK90}
outlined a different approach to finite characters. Their approach was not
based on total positivity and theory of functions of a complex variable, as
Edrei's and Thoma's. Instead, it relied on the ideas of discrete potential
theory and combinatorics of symmetric functions. In a broader context this
approach was described in detail in Kerov, Okounkov, and Olshanski \cite{KOO98}
and Okounkov and Olshanski \cite{OO98}, where the character problem was
rephrased in the language of boundaries of two infinite graphs, the Young graph
$\Y$ and the Gelfand--Tsetlin graph $\GT$. These are two model examples of the
so-called branching graphs; they encode branching rules of the irreducible
characters of symmetric and unitary groups, respectively.

Denote by $\chi_\nu$ the irreducible character of $S(n)$ or $U(N)$. Here index
$\nu$ is either a Young diagram with $n$ boxes or a signature of length $N$ (a
highest weight for $U(N)$). In Vershik--Kerov's approach, one studies the
limiting behavior of the normalized characters
$$
\wt\chi_\nu:=\frac{\chi_\nu}{\chi_\nu(e)}
$$
when $n$ or $N$ becomes large, and the diagram/signature is $n$ or $N$ dependent.
It turns out that possible limits of $\wt\chi_\nu$ are exactly the finite characters of
$S(\infty)$ or $U(\infty)$, respectively.

\subsection{A combinatorial formulation}

In the language of branching graphs, the question of asymptotics of
$\wt\chi_\nu$ can be reformulated in a purely combinatorial fashion. More
exactly, one asks about the asymptotics of
\begin{gather}
\frac{\dim(\ka,\nu)}{\dim\nu} \qquad \textrm{(for the symmetric group)}
\label{eqH.1}\\
\frac{\Dim_{K,N}(\ka,\nu)}{\Dim_N\nu} \qquad \textrm{(for the unitary group)},
\label{eqH.2}
\end{gather}
with the following notations:

$\bullet$ In the symmetric group case, $\ka$ is a Young diagram with $k<n$
boxes; $\dim\nu$ is the number of standard Young tableaux of shape $\nu$;
$\dim(\ka,\nu)$ is the number of the standard tableaux of skew shape $\nu/\ka$ if
$\ka\subset\nu$, and 0 if $\ka\not\subset\nu$.

$\bullet$ In the unitary group case, $\ka$ is a signature of length $K<N$;
$\Dim_N\nu$ is the number of triangular Gelfand--Tsetlin schemes with fixed top
row $(\nu_1,\dots,\nu_N)$; $\Dim_{K,N}(\ka,\nu)$ is the number of truncated
(trapezoidal) Gelfand--Tsetlin schemes with top row $\nu$ and bottom row $\ka$.

The ``dimensions'' $\dim$ and $\Dim$ count certain finite sets of monotone paths in $\Y$ and
$\GT$. In both cases, the problem consists in classification of all possible ways
for $\nu$ to approach infinity so that the  ``relative dimension'' \eqref{eqH.1} or
\eqref{eqH.2} has a finite limit for any fixed diagram/signature $\ka$.
These possibilities are parameterized by the points of the branching graph's boundary.

Let us note that the denominator in \eqref{eqH.1} or \eqref{eqH.2} is given by a
relatively simple formula, while computing the numerator is substantially harder.
This basic difficulty results in nontriviality of the asymptotic analysis.

\subsection{Motivations}

In the present paper we return to the problem of finding the boundary of
$\GT$ and obtain a new proof of completeness of the list of characters
of $U(\infty)$ given by Voiculescu in \cite{Vo76}. The reader would be fully
justified to ask why we decided to reconsider an old theorem and produce its
third proof, especially since our proof is not that simple. Here are our
arguments.

\medskip

(a) The boundary descriptions for $\Y$ and $\GT$ are strikingly similar. In
terms of total positivity, the points of both boundaries correspond to infinite
totally positive Toeplitz matrices; in the first case the matrices grow to one
side (have format $\mathbb N\times\mathbb N$), and in the second case they grow
to both sides (have format $\mathbb Z\times\mathbb Z$). We are confident that
the parallelism between $\Y$ and $\GT$ is deeply rooted, and one should expect
its appearance in other aspects as well \footnote{New results in this direction
are contained in our paper \cite{BO11}.}. However, if one compares the proofs
given in \cite{KOO98} for $\Y$ and in \cite{OO98} for $\GT$ then one would
notice that they are substantially different.

More exactly, in the case of $\Y$ in \cite[Theorem 8.1]{OO97a} the authors
obtained a formula that expressed the relative dimension \eqref{eqH.1} through
the shifted Schur functions. This formula is well adapted for the asymptotic
analysis, and \cite{KOO98} was based on this formula (and on its generalization
that includes the Jack parameter \cite{OO97b}). In the case of $\GT$ there was
no analogous expression for \eqref{eqH.2}. For that reason the authors of
\cite{OO98} had to follow a round-about path inspired by an idea from
\cite{VK82} of pursuing the asymptotics of the Taylor expansion of characters
$\wt\chi_\nu$ at the unit element of the group.

In the present paper, in contrast to \cite{OO98}, we work directly with the
relative dimension \eqref{eqH.2} and derive a formula for it that is suitable
for asymptotic analysis. Hence, we achieve uniformity in the asymptotic
approach to the boundaries of $\Y$ and $\GT$.

\medskip

(b) We believe that our formula for the relative dimension \eqref{eqH.2} and
its proof are of independent interest in algebraic combinatorics. The formula
involves certain new symmetric functions of Schur type. In the proof we also
use the so-called dual Schur functions that were thoroughly investigated in a
recent paper by Molev \cite{Mo09}.

\medskip

(c) The description of $\GT$'s boundary is derived below from a new result that
we call the Uniform Approximation Theorem. It is important to us as it allows
to substantially strengthen our results in \cite{BO10} on Markov dynamics on
the boundary of $\GT$. This development will be described in a separate
publication.

\medskip

(d)  To conclude, we believe that the classification of finite characters
of $U(\infty)$ is a difficult and truly deep result, and already for that reason
its third proof should not be dismissed as excessive.

\medskip

\subsection{Organization of the paper}

Section 2 contains main definitions and auxiliary results. A part of those is
contained in one form or another in Voiculescu \cite{Vo76}. At the end of the
section we give a description of $\GT$'s boundary (Theorem \ref{thmF.1} and
Corollary \ref{corF.1}).

In Section 3 we state the Uniform Approximation Theorem (Theorem
\ref{thmG.1}) and explain how it implies the results on the boundary of $\GT$.

Sections 4--8 contain the proof of the Uniform Approximation Theorem.

In Section 4 we give an auxiliary result on an identity of Cauchy type; here dual
Schur functions come about (more general results in this direction can be found in Molev
\cite{Mo09}).

In Section 5 we prove a different identity of Cauchy type:
\begin{equation}\label{eqH.3}
H^*(t_1;\nu)\dots H^*(t_K;\nu)
=\sum_{\ka\in\GT_K}\frac{\Dim_{K,N}(\ka,\nu)}{\Dim_N\nu}\,\Sf_{\ka\mid
N}(t_1,\dots,t_K).
\end{equation}
Here $t_1,\dots,t_K$ are complex variables, $\nu$ is an arbitrary signature
of length $N>K$,
\begin{equation*}
H^*(t;\nu)=\prod_{i=1}^N\frac{t+i}{t+i-\nu_i},
\end{equation*}
the summation in the right-hand side of \eqref{eqH.3} is over signatures $\ka$ of length
$K$, and $\Sf_{\ka\mid N}(t_1,\dots,t_K)$ are certain new analogs of Schur functions in
$K$ variables. The coefficients in front of these functions are the relative dimensions
\eqref{eqH.2} that we are interested in.

In Section 6 we show how \eqref{eqH.3} implies a Jacobi--Trudi type formula for
the relative dimension. It expresses the relative dimension as a determinant of
size $K\times K$ whose matrix elements are coefficients of the decomposition of
$H^*(t;\nu)$ on certain rational functions.

Section 7 explains how to write those coefficients through residues of
$H^*(t;\nu)$. As a result, we obtain an explicit formula for the relative dimension
(Theorem \ref{thmD.1}). For comparison, we also give a different formula (Remark \ref{remG.1}).
In contrast to Theorem \ref{thmD.1}, its derivation is simple but the formula seems
useless for our purposes.

In Section 8 using Theorem \ref{thmD.1} we conclude the proof of the Uniform Approximation
Theorem.

Together with the Uniform Approximation Theorem, the formula of Theorem \ref{thmD.1}
is one of our main results. It is plausible that this formula can be obtained in
a simpler way, and we would be very interested in seeing how to do that.
It often happens that combinatorial identities have different proofs which can
be simpler than the original derivation.  (For example, one could try
to derive Theorem \ref{thmD.1} from the formula of Remark
\ref{remG.1} or from the binomial formula for the normalized characters
$\wt\chi_\nu$ that \cite{OO97a} was based upon.)

Not only the Uniform Approximation Theorem provides a new derivation of $\GT$'s
boundary, but it also immediately implies the main results of \cite{VK82} and
\cite{OO98} on large $N$ asymptotics of the normalized characters
$\wt\chi_\nu$. In the last Section 9 we demonstrate that conversely, the
Uniform Approximation Theorem is not hard to prove using the results of
\cite{OO98} if one additionally employs the log-concavity of characters
$\wt\chi_\nu$ discovered by Okounkov \cite{Ok97}. We emphasize however that
this approach gives nothing for Theorem \ref{thmD.1}.

Let us finally mention a recent paper by Gorin \cite{Go11} where the boundary
of a ``$q$--analog'' of $\GT$ was described (the edges of the graph are supplied with
certain formal $q$-dependent multiplicities). It would be interesting to extend
the approach of the present paper to the $q$--$\GT$ case.

\subsection*{Acknowledgements} A.~B. was partially supported by NSF-grant
DMS-1056390. G.~O. was partially supported by a grant from Simons Foundation
(Simons--IUM Fellowship), the RFBR-CNRS grant 10-01-93114, and the project SFB
701 of Bielefeld University.

\section{Preliminaries}\label{F}

\subsection{The graph $\GT$}\label{sectF.2}
Following \cite{Wey39}, for $N\ge 1$ define a {\it signature\/} of length $N$
as an $N$-tuple of nonincreasing integers $\nu=(\nu_1\ge\dots\ge \nu_N)$, and
denote by $\GT_N$ the set of all such signatures.

Two signatures $\la\in\GT_{N-1}$ and $\nu\in \GT_{N}$ {\it interlace\/} if
$\nu_{i+1}\le \la_i\le\nu_i$ for all meaningful values of indices; in this case
we write $\la\prec\nu$.

Let $\GT=\bigsqcup_{N\ge 1} \GT_N$ be the set of signatures of arbitrary
length, and equip $\GT$ with edges by joining $\la$ and $\nu$ iff $\la\prec\nu$
or $\nu\prec\la$. This turns $\GT$ into a graph that is called the {\it
Gelfand--Tsetlin graph\/}. We call $\GT_N\subset\GT$ the {\it level\/ $N$
subset\/} of the graph.

By a {\it path\/} between two vertices $\ka\in\GT_K$ and $\nu\in\GT_N$, $K<N$,
we mean a sequence
$$
\ka=\la^{(K)}\prec\la^{(K+1)}\prec\dots\prec\la^{(N)}=\nu\in \GT_N.
$$
Such a path can be viewed as an array of numbers
$$
\bigl\{\la^{(j)}_i\bigr\}, \quad  K\le j\le N, \quad 1\le i\le j,
$$
satisfying the inequalities $\la^{(j+1)}_{i+1}\le \la_i^{(j)}\le
\la^{(j+1)}_i$. It is called a {\it Gelfand--Tsetlin scheme\/}. If $K=1$, the
scheme has triangular form and if $K>1$, it has trapezoidal form.

Let $\Dim_{K,N}(\ka,\nu)$ denote the number of paths between $\ka$ and $\nu$,
and let $\Dim_N\nu$ be the number of all paths starting at an arbitrary vertex
of level 1 and ending at $\nu$. Both these numbers are always finite; note that
they count the lattice points in some bounded convex polyhedra. The number
$\Dim_{K,N}(\ka,\nu)$ may be equal to 0, but $\Dim_N\nu$ is always strictly
positive.

For $N\ge 2$ denote by $\La^N_{N-1}$ the matrix of format $\GT_N\times
\GT_{N-1}$ with the entries
$$
\La^N_{N-1}(\nu,\la)=\begin{cases}\dfrac{\Dim_{N-1}\la}{\Dim_N\nu}, &
\la\prec\nu,\\
0, & \textrm{otherwise}.
\end{cases}
$$
By the very definition of the $\Dim$ function,
$$
\Dim_N\nu=\sum_{\la:\, \la\prec\nu}\Dim_{N-1}\la.
$$
It follows that $\La^N_{N-1}$ is a stochastic matrix:
$$
\sum_{\la\in\GT_{N-1}}\La^N_{N-1}(\nu,\la)=1 \qquad \forall \nu\in\GT_N.
$$

More generally, for $N>K\ge1$, the matrix product
\begin{equation}\label{eqF.15}
\La^N_K:=\La^N_{N-1}\La^{N-1}_{N-2}\dots\La^{K+1}_K
\end{equation}
is a stochastic matrix, too, and its entries are
$$
\La^N_K(\nu,\ka)=\frac{\Dim_K\ka\,\Dim_{K,N}(\ka,\nu)}{\Dim_N\nu}.
$$

\subsection{The boundary of $\GT$}
We say that an infinite sequence $M_1,M_2,\dots$ of probability distributions
on the sets $\GT_1,\GT_2,\dots$, respectively, forms a {\it coherent system\/}
if the distributions are consistent with the transition matrices $\La^2_1,
\La^3_2,\dots$, meaning that
$$
M_N\La^N_{N-1}=M_{N-1} \qquad \forall N\ge2.
$$
Here we interpret $M_N$ as a row vector $\{M_N(\nu): \nu\in\GT_N\}$, which
makes it possible to define the multiplication in the left-hand side. In more
detail, the relation means
$$
\sum_{\nu\in\GT_N}M_N(\nu)\La^N_{N-1}(\nu,\la)=M_{N-1}(\la) \qquad
\forall\la\in\GT_{N-1}.
$$
Note that the set of all coherent systems is a convex set: if
$\{M_N:N=1,2,\dots\}$ and $\{M'_N:N=1,2,\dots\}$ are two coherent systems, then
for any $p\in[0,1]$, the convex combination $\{pM_N+(1-p)M'_N:N=1,2,\dots\}$ is
a coherent system, too. A coherent system is said to be {\it extreme\/} if it
is an extreme point in this convex set.

\begin{definition}
The {\it boundary\/ $\pd(\GT)$ of the Gelfand--Tsetlin graph\/} $\GT$ is
defined as the set of extreme coherent systems of distributions on $\GT$.
\end{definition}

This definition mimics the well-known definition of the {\it minimal part of
the Martin entrance boundary\/} of a Markov chain (see, e.g. \cite{KSK76}).
Indeed, consider the infinite chain
\begin{equation}\label{eqF.1}
\GT_1\dashleftarrow\GT_2\dashleftarrow\GT_2\dashleftarrow\dots
\end{equation}
where the dashed arrows symbolize the transition matrices $\La^N_{N-1}$. One
may regard \eqref{eqF.1} as a Markov chain with time parameter $N=1,2,\dots$
ranging in the reverse direction, from infinity to 1, and with the state space
varying with time. Although such a Markov chain looks a bit unusual, the
conventional definition of the minimal entrance boundary can be adapted to our
context, and this leads to the same space $\pd(\GT)$. Note that the minimal
entrance boundary may be a proper subset of the whole Martin entrance boundary,
but for the concrete chain \eqref{eqF.1} these two boundaries coincide.

One more interpretation of the boundary $\pd(\GT)$ is the following: it
coincides with the {\it projective limit\/} of chain \eqref{eqF.1} in the
category whose objects are measurable spaces and morphisms are defined as
Markov transition kernels (stochastic matrices are just simplest instances of
such kernels).

For more detail about the concept of entrance boundary employed in the present
paper, see, e.g., \cite{Dy71}, \cite{Dy78}, \cite{Wi85}.

\subsection{Representation-theoretic interpretation}
Let $U(N)$ denote the group of $N\times N$ unitary matrices or, equivalently,
the group of unitary operators in the coordinate space $\C^N$. For every
$N\ge2$ we identify the group $U(N-1)$ with the subgroup of $U(N)$ that fixes
the last basis vector. In this way we get an infinite chain of groups embedded
into each other
\begin{equation}\label{eqF.2}
U(1)\subset U(2)\subset U(3)\subset\dots
\end{equation}

As is well known, signatures  from $\GT_N$ parameterize irreducible characters
of $U(N)$; given $\nu\in\GT_N$, let $\chi_\nu$ denote the corresponding
character. The {\it branching rule\/} for the irreducible characters of the
unitary groups says that
\begin{equation}\label{eqF.3}
\chi_\nu\Big|_{U(N-1)}=\sum_{\la:\,\la\prec\nu}\chi_\la \qquad \forall
\nu\in\GT_N, \quad N\ge2,
\end{equation}
where the vertical bar means the restriction map from $U(N)$ to $U(N-1)$. The
graph $\GT$ just reflects the rule \eqref{eqF.3}; for this reason one says that
$\GT$ is the {\it branching graph\/} for the characters of the unitary groups.

It follows from \eqref{eqF.3} that $\Dim_N\nu$ equals $\chi_\nu(e)$, the value
of $\chi_\nu$ at the unit element of $U(N)$, which is the same as the dimension
of the corresponding irreducible representation. This explains our notation.

Let $U(\infty)$ be the union of the groups \eqref{eqF.2}. Although $U(\infty)$
is not a compact group, one can develop for it a rich theory of characters
provided that the very notion of character is suitably revised:

\begin{definition}
By a {\it character\/} of $U(\infty)$ we mean a function $\chi:U(\infty)\to\C$
satisfying the following conditions:

$\bullet$ $\chi$ is continuous in the inductive limit topology on $U(\infty)$
(which simply means that the restriction of $\chi$ to every subgroup $U(N)$ is
continuous);

$\bullet$ $\chi$ is a class function, that is, constant on conjugacy classes;

$\bullet$ $\chi$ is positive definite;

$\bullet$ $\chi(e)=1$.

Next, observe that the set of all characters is a convex set and say that
$\chi$ is an {\it extreme character\/} if it is an extreme point of this set.

\end{definition}

The above definition makes sense for any topological group. In particular, the
extreme characters of $U(N)$ are precisely the {\it normalized\/} irreducible
characters
$$
\wt\chi_\nu:=\frac{\chi_\nu}{\chi_\nu(e)}=\frac{\chi_\nu}{\Dim_N\nu}, \qquad
\nu\in\GT_N,
$$
and the set of all characters of $U(N)$ is an infinite-dimensional simplex; its
vertices are the characters $\wt\chi_\nu$.

The extreme characters of $U(\infty)$ can be viewed as analogs of
characters $\wt\chi_\nu$.

The representation-theoretic meaning of the extreme characters is that they
correspond to {\it finite factor representations\/} of $U(\infty)$; see
\cite{Vo76}.

\begin{proposition}\label{propF.1}
There is a natural bijective correspondence between the  characters of the
group $U(\infty)$ and the coherent systems on the graph\/ $\GT$, which also
induces a bijection between the extreme characters and the points of the
boundary $\pd(\GT)$.
\end{proposition}

\begin{proof}
If $\chi$ is a character of $U(\infty)$, then for every $N=1,2, \dots$ the
restriction $\chi$ to $U(N)$ is a convex combination of normalized characters
$\wt\chi_\nu$. The corresponding coefficients, say $M_N(\nu)$, are nonnegative
and sum to 1, so that they determine a probability distribution $M_N$ on
$\GT_N$. Further, the family $\{M_N:N=1,2,\dots\}$ is a coherent system. The
correspondence $\chi\to \{M_N\}$ defined in this way is a bijection of the set
of characters of $U(\infty)$ onto the set of coherent systems, which is also an
isomorphism of convex sets. This entails a bijection between the extreme points
of the both sets, that is, the extreme characters and and the points of
$\pd(\GT)$.

For more detail, see \cite{Ols03} and especially Proposition 7.4 therein.
\end{proof}

Informally, Proposition \ref{propF.1} says that the chain \eqref{eqF.1} is dual
to the chain \eqref{eqF.2} and the boundary $\pd(\GT)$ is a kind of dual object
to $U(\infty)$.

\subsection{The space $\Om$ and the function $\Phi(u;\om)$}
Let $\R_+\subset\R$ denote the set of nonnegative real numbers, $\R_+^\infty$
denote the product of countably many copies of $\R_+$, and set
$$
\R_+^{4\infty+2}=\R_+^\infty\times\R_+^\infty\times\R_+^\infty\times\R_+^\infty
\times\R_+\times\R_+.
$$
Let $\Om\subset\R_+^{4\infty+2}$ be the subset of sextuples
$$
\om=(\al^+,\be^+;\al^-,\be^-;\de^+,\de^-)
$$
such that
\begin{gather*}
\al^\pm=(\al_1^\pm\ge\al_2^\pm\ge\dots\ge 0)\in\R_+^\infty,\quad
\be^\pm=(\be_1^\pm\ge\be_2^\pm\ge\dots\ge 0)\in\R_+^\infty,\\
\sum_{i=1}^\infty(\al_i^\pm+\be_i^\pm)\le\de^\pm, \quad \be_1^++\be_1^-\le 1.
\end{gather*}
Equip $\R_+^{4\infty+2}$ with the product topology. An important fact is that,
in the induced topology, $\Om$ is a locally compact space. Moreover, it is
metrizable and separable. Any subset in $\Om$ of the form
$\de^++\de^-\le\const$ is compact, which shows that a sequence of points in
$\Om$ goes to infinity if and only the quantity $\de^++\de^-$ goes to infinity.

Set
$$
\ga^\pm=\de^\pm-\sum_{i=1}^\infty(\al_i^\pm+\be_i^\pm)
$$
and note that $\ga^+,\ga^-$ are nonnegative. For $u\in\C^*$ and $\om\in\Om$ set
\begin{equation}\label{eqF.11}
\Phi(u;\om)= e^{\ga^+(u-1)+\ga^-(u^{-1}-1)}
\prod_{i=1}^\infty\frac{1+\be_i^+(u-1)}{1-\al_i^+(u-1)}
\,\frac{1+\be_i^-(u^{-1}-1)}{1-\al_i^-(u^{-1}-1)}.
\end{equation}

Here are some properties of $\Phi(u;\om)$ as a function in variable $u$:

For any fixed $\om$, this is a meromorphic function in $u\in\C^*$ with poles on
$(0,1)\cup(1,+\infty)$. The poles do not accumulate to $1$, so that the
function is holomorphic in a neighborhood of the unit circle
$$
\T:=\{u\in\C: |u|=1\}.
$$
Obviously,
$$
\Phi(1;\om)=1 \qquad \forall\om\in\Om.
$$
In particular, $\Phi(u;\om)$ is well defined and continuous on $\T$.

\begin{proposition}\label{propF.6} One has
$$
|\Phi(u;\om)|\le 1 \quad \textrm{for $u\in\T$}.
$$
\end{proposition}

\begin{proof}
Indeed, the claim actually holds for every factor in \eqref{eqF.11}:
\begin{equation}\label{eqF.12}
|(1-\alpha_i^\pm (u^{\pm 1}-1))^{-1}|\le1, \qquad |1+\beta_i^\pm (u^{\pm
1}-1)|\le1,\qquad |e^{\gamma^\pm(u^{\pm 1}-1)}|\le1.
\end{equation}
\end{proof}

\begin{proposition}\label{propF.3}
Different $\om$'s correspond to different functions $\Phi(\,\cdot\,,\om)$.
\end{proposition}

\begin{proof}
See \cite[Proof of Theorem 5.1, Step 3]{OO98}. Here the condition
$\be_1^++\be_1^-\le1$ plays the decisive role.
\end{proof}

\begin{proposition}\label{propF.5}
There exists a homeomorphism $S:\Om\to\Om$ such that
$$
\Phi(u;S\om)=u\Phi(u;\om).
$$
\end{proposition}

\begin{proof}
Indeed, observe that
$$
u(1+\be(u^{-1}-1))=1+(1-\be)(u-1).
$$
It follows that $S$ has the following form: it deletes $\be^-_1$ from the list
of the $\be^-$-coordinates of $\om$ (so that $\be^-_2$ becomes coordinate number 1,
$\be^-_3$ becomes coordinate number 2, etc.) and adds a new $\be^+$ coordinate equal to
$1-\be^-_1$. Note that this new coordinate is $\ge\be^+_1$ (due to the
condition $\be^+_1+\be^-_1\le1$), so that it acquires number 1, $\be^+_1$
becomes coordinate number 2, etc. All the remaining coordinates remain intact.
\end{proof}

\subsection{The functions $\varphi_\nu(\om)$ and the Markov kernels
$\La^\infty_N$}\label{sectF.1}

Since $\Phi(\,\cdot\,;\om)$ is regular in a neighborhood of $\T$, it can be
expanded into a Laurent series:
$$
\Phi(u;\om)=\sum_{n=-\infty}^\infty \varphi_n(\om)u^n,
$$
where
\begin{equation}\label{eqF.10}
\varphi_n(\omega)=\frac 1{2\pi i}\oint_\T \Phi(u;\om)\frac {du}{u^{n+1}},\qquad
n\in\Z.
\end{equation}
The Laurent coefficients $\varphi_n(\om)$ play an important role in what
follows.

More generally,  we associate with every $\nu\in\GT_N$, $N=1,2,\dots$, the
following function on $\Om$
$$
\varphi_\nu(\om):=\det[\varphi_{\nu_i-i+j}(\om)]_{i,j=1}^N.
$$

Recall that the {\it rational Schur function\/} in $N$ variables is the Laurent
polynomial indexed by a signature $\nu\in\GT_N$ and defined by
$$
S_\nu(u_1,\dots,u_N)=\frac{\det\bigl[u_i^{\nu_j+N-j}\bigr]_{i,j=1}^N}{\prod\limits_{1\le
i<j\le N}(u_i-u_j)}\,.
$$
These Laurent polynomials form a basis in the algebra
$\C[u_1^{\pm1},\dots,u_N^{\pm1}]^\sym$ of symmetric Laurent polynomials.

\begin{proposition}
For $N=1,2,\dots$ the following expansion holds
\begin{equation}\label{eqF.5}
\Phi(u_1;\om)\dots\Phi(u_N;\om)=\sum_{\nu\in\GT_N}\varphi_\nu(\om)
S_\nu(u_1,\dots,u_N),
\end{equation}
where the series converges in a neighborhood of\/ $\T\subset\C^*$ depending on
$\om\in\Om$.
\end{proposition}

\begin{proof}
This is a very simple but fundamental fact. See, e.g., \cite[Lemme 2]{Vo76}.
\end{proof}

Note that if \eqref{eqF.5} is interpreted as an identity of formal series,
without addressing the question of convergence, then the result holds for an
arbitrary two-sided infinite formal power series in $u$ in place of
$\Phi(u;\om)$. Further, if the series is expanded on nonnegative powers of $u$
only and the constant term is equal to 1, then the product in the left-hand
side may be infinite. In that case the right-hand side becomes an expansion on Schur
symmetric functions in infinitely many variables, indexed by arbitrary
partitions. See, e.g., \cite[pp. 99-100]{Li50}.

\begin{proposition}\label{propF.7}
The functions $\varphi_\nu(\om)$ are nonnegative.
\end{proposition}

\begin{proof}
See \cite[Proposition 2]{Vo76}.
\end{proof}

This fine property means that for any $\om$, all minors of the two-sided
infinite Toeplitz matrix $[\varphi_{i-j}(\om)]_{i,j\in\Z}$ extracted from
several consecutive columns are nonnegative. But this actually implies that
{\it all\/} minors are nonnegative (see \cite[p.218]{Boy83}). That is, the
two-sided infinite sequence $\{\varphi_n(\om)\}_{n\in\Z}$ is {\it totally
positive\/}.

As is well known, the Laurent polynomials $S_\nu$ with $\nu\in\GT_N$ determine
the irreducible characters of $U(N)$ in the sense that $\chi_\nu=S_\nu$ on the
torus $\T^N=\T\times\dots\times\T$ identified with the subgroup of diagonal
matrices in $U(N)$. It follows that
\begin{equation}\label{eqF.6}
S_\nu(\,\underbrace{1,\dots,1}_{N}\,)=\Dim_N\nu
\end{equation}
and, more generally,
\begin{equation}\label{eqF.7}
S_\nu(u_1,\dots,u_K,\,\underbrace{1,\dots,1}_{N-K}\,)
=\sum_{\ka\in\GT_K}\Dim_{K,N}(\ka,\nu)S_\ka(u_1,\dots,u_K), \qquad K<N.
\end{equation}

Equalities \eqref{eqF.5}-\eqref{eqF.7} imply

\begin{proposition}\label{propF.2}
Set
\begin{equation}\label{eqF.4}
\La^\infty_N(\om,\nu)=\Dim_N\nu\cdot\varphi_\nu(\om),
\end{equation}
where $N=1,2,\dots$, $\om\in\Om$, and $\nu\in\GT_N$.

{\rm(i)} $\La^\infty_N$ is a Markov kernel, that is,
$\La^\infty_N(\om,\nu)\ge0$ for all $\om$ and $\nu$, and
\begin{equation}\label{eqF.8}
\sum_{\nu\in\GT_N}\La^\infty_N(\om,\nu)=1.
\end{equation}

{\rm(ii)} For $N>K\ge1$ there holds
\begin{equation}\label{eqF.9}
\La^\infty_N\La^N_K=\La^\infty_K.
\end{equation}
Or, in more detail,
\begin{equation}\label{eqF.16}
\sum_{\nu\in\GT_N}\La^\infty(\om,\nu)\La^N_K(\nu,\ka)=\La^\infty_K(\om,\ka),
\qquad \forall \om\in\Om, \quad \forall\ka\in\GT_K.
\end{equation}
\end{proposition}

\begin{proof}
The property $\La^\infty_N(\om,\nu)\ge0$ is ensured by Proposition
\ref{propF.7}.

Plug in $u_1=\dots=u_N=1$ into \eqref{eqF.5} and use the fact that
$\Phi(1;\om)=1$. Then, because of \eqref{eqF.6}, we get \eqref{eqF.8}.

Likewise, plug in $u_{K+1}=\dots=u_N=1$ into \eqref{eqF.5} and apply
\eqref{eqF.7}. Comparing the result with the expansion
$$
\Phi(u_1;\om)\dots\Phi(u_K;\om)=\sum_{\ka\in\GT_K}\varphi_{\ka}(\om)
S_\ka(u_1,\dots,u_K)
$$
we get \eqref{eqF.16}.

\end{proof}

\subsection{The Feller property}

For a locally compact metrizable separable space $X$, denote by $C_0(X)$ the
space of real-valued continuous functions vanishing at infinity. This is a
separable Banach space with respect to the supremum norm. In particular, the
definition makes sense for $X=\Om$ and also for $\GT_N$, since this a countable
discrete space. Let us interpret functions $f\in C_0(\GT_N)$ as column vectors.

\begin{proposition}\label{propF.4}
The functions $\varphi_n(\om)$, $n\in\Z$, are continuous functions on $\Om$
vanishing at infinity.
\end{proposition}

An immediate consequence of this result is the following

\begin{corollary}\label{corF.3}
For every $N=1,2,\dots$ the Markov kernel $\La^\infty_N$ is a Feller kernel,
meaning that the map $f\mapsto\La^\infty_Nf$ is a continuous {\rm(}actually,
contractive{\rm)} linear operator $C_0(\GT_N)\to C_0(\Om)$.
\end{corollary}

\begin{proof}[Proof of the corollary]
It follows from the proposition and the definition of the kernel that for
$\nu\in\GT_N$ fixed, the function $\om\mapsto\La^\infty_N(\om,\nu)$ is
continuous and vanishes at infinity. This is equivalent to the Feller property.
\end{proof}

\begin{proof}[Proof of the proposition]
The continuity is established in \cite[Proof of Theorem 8.1, Step
1]{Ols03}.

Now we must prove that for any fixed $n\in\Z$ and any sequence of points
$\{\om(k)\}$ in $\Om$ converging to infinity one has
$\lim_{k\to\infty}\varphi_n(\om(k))=0$. It is enough to prove a weaker claim
that the same limit relation holds for a subsequence in $\{\om(k)\}$. Below we
write $\al^{\pm}_i(k)$, $\be^{\pm}_i(k)$, $\de^\pm(k)$ for the coordinates of
$\om(k)$.

\medskip

\noindent{\it Step 1.\/} We may assume that $\sup_{k\ge 1}
\alpha^\pm_1(k)<\infty$. Indeed, if there is a subsequence $\{k_m\}_{m\ge 1}$
such that $\alpha_1^\pm(k_m)\to\infty$, then along this subsequence
$(1-\alpha_1^\pm (u^{\pm 1}-1))^{-1}$ tends to zero uniformly on any compact
subset of $\T\setminus\{u=1\}$, which implies that the right-hand side of
\eqref{eqF.10} tends to zero.

Let us fix $A>0$ such that $\sup_{k}\alpha_1^{\pm}(k)\le A$.

\medskip

\noindent{\it Step 2.\/} Assume $\omega$ ranges over the subset of elements of
$\Omega$ with $\alpha_1^\pm\le A$ and $\beta_1^\pm\le \frac 12$. Then for any
$\epsilon>0$,
$$
\lim_{\delta^++\delta^-\to\infty}\Phi(u;\om)=0\ \textrm{ uniformly on }\
\{u\in\T,\, \Re u\le 1-\epsilon\}.
$$

Indeed, assume $0\le\be\le\frac12$ and $0\le\al\le A$. For $u$ on the unit
circle with $\Re u\le 1-\epsilon$ we have elementary estimates
\begin{multline*}
|1+\beta(u-1)|^2= (1-\beta)^2+\beta^2+2\beta(1-\beta)\Re
u\\=1-2\beta(1-\beta)(1-\Re u)\le 1-2\beta(1-\beta)\epsilon\le
1-\beta\epsilon\le e^{-\beta\epsilon},
\end{multline*}

\begin{multline*}
|1-\alpha(u-1)|^{-2}=(1+2\alpha(1+\alpha)(1-\Re u))^{-1}\\\le
(1+2\alpha(1+\alpha)\epsilon)^{-1}\le (1+2\alpha\epsilon)^{-1}\le
e^{-\const\alpha\epsilon}
\end{multline*}
with a suitable constant $\const>0$ that depends only on $A$,

$$
|e^{\gamma^+(u-1)+\gamma^-(u^{-1}-1)}|^2=e^{-2(\gamma^++\gamma^-)(1-\Re u)}\le
e^{-2(\gamma^++\gamma^-)\epsilon}.
$$
Thus, if
$$
\delta^++\delta^-=\gamma^++\gamma^-+
\sum_{i=1}^\infty(\alpha^+_i+\beta^+_i+\alpha^-_i+\beta^-_i)\to\infty
$$
then at least one of the right-hand sides in these estimates yields an
infinitesimally small contribution, and consequently $\Phi(u;\om)$ must be
small.

Thus, under the above assumptions on $\omega$, we see that $\omega\to\infty$
implies $\varphi_n(\omega)\to0$ {\it uniformly\/} on $n\in\Z$.

\medskip

\noindent {\it Step 3.\/} Now we get rid of the restriction
$\beta_1^{\pm}\le\frac12$. Set
$$
B^\pm(k)=\#\{i\ge 1\mid \beta_i^\pm(k)>\tfrac 12\}.
$$
Since for any $k\ge 1$ we have $\beta_1^+(k)+\beta_1^-(k)\le 1$, at least one
of the numbers $B^\pm(k)$ is equal to 0. For inapplicability of the Step 2 argument,
for any subsequence $\{\omega_{k_m}\}$ of our sequence $\{\omega(k)\}$, we must
have $B^+(k_m)+B^-(k_m)\to\infty$. Hence, possibly passing to a subsequence and
switching $+$ and $-$, we may assume that $B^+(k)\to\infty$ as $k\to\infty$.

Set
$$
\tilde\om(k):=S^{-B^+(k)}\om(k),
$$
where $S$ is the homeomorphism from Proposition \ref{propF.5}. In words,
$\tilde\omega(k)$ is obtained from $\omega(k)$ as follows: Each
$\beta^+$-coordinate of $\omega(k)$ that is $>1/2$ is transformed into a
$\beta^-$ coordinate of $\tilde \omega(k)$ equal to $1$ minus the original
$\beta^+$-coordinate; all other coordinates are the same (equivalently, the
function $\Phi(u;\om))$ is multiplied by $u^{-B^+(k)}$). Let
$(\tilde\alpha^\pm(k),\tilde\beta^\pm(k),\tilde\gamma^\pm(k),\tilde\delta^\pm(k))$
be the coordinates of $\tilde \omega(k)$.

\medskip

\noindent {\it Step 4.\/} Since no $\beta$-coordinates of $\tilde\omega(k)$ are
greater than $1/2$, the result of Step 2 implies that if
$\sup(\tilde\delta^+(k)+\tilde\delta^-(k))=\infty$ then
$\varphi_n(\omega(k))=\varphi_{n-B^+(k)}(\tilde\omega(k))\to 0$ as $k\to\infty$
along an appropriate subsequence (because the conclusion of that step holds
uniformly on $n\in\Z$). Hence, it remains to examine the case when
$\tilde\delta^+(k)+\tilde\delta^-(k)$ is bounded.

Let us deform the integration contour in \eqref{eqF.10} to $|u|=R$ with
$A/(1+A)<R<1$. Using the estimates (for $|u|=R$, $0\le \alpha\le A$, $0\le
\beta\le \frac12$)
$$
\gathered
|1+\beta(u^{\pm 1}-1)|\le 1+\beta|u^{\pm 1}-1|\le e^{\const_1\beta},\\
|1-\alpha(u^{\pm 1}-1)|^{-1}\le |1-\alpha(R^{\pm 1}-1)|^{-1}\le e^{\const_2\alpha},\\
|e^{\gamma(u^{\pm 1}-1)}|\le e^{\const_3\gamma}
\endgathered
$$
with suitable $\const_j>0$, $j=1,2,3$, we see that
$$
|\Phi(u;\tilde\omega(k))|\le
e^{\const_4({\tilde\delta^+(k)+\tilde\delta^-(k)})}
$$
for a $\const_4>0$, which remains bounded.

On the other hand, as $k\to\infty$, the factor $u^{-n-1+B^+(k)}$ in the
integral representation \eqref{eqF.10} for
$\varphi_{n-B^+(k)}(\tilde\omega(k))$ tends to 0 uniformly in $u$, because
$B^+(k)\to+\infty$ and $|u|=R<1$. Hence,
$\varphi_{n}(\omega(k))=\varphi_{n-B^+(k)}(\tilde\omega(k))\to 0$ as
$k\to\infty$, and the proof of the proposition is complete.
\end{proof}

The following proposition is an analog of Corollary \ref{corF.3} for the
stochastic matrices $\La^N_K$. It is much easier to prove.

\begin{proposition}\label{propF.9}
Let $K<N$. If $\ka\in\GT_K$ is fixed and $\nu$ goes to infinity in the
countable discrete space $\GT_N$, then $\La^N_K(\nu,\ka)\to0$. Equivalently,
the map $f\mapsto \La^N_K f$ is a continuous {\rm(}actually contractive{\rm)}
operator $C_0(\GT_K)\to C_0(\GT_N)$, so that $\La^N_K$ is Feller.
\end{proposition}

\begin{proof}
Because of \eqref{eqF.15} it suffices to prove the assertion of the proposition
in the particular case when $K=N-1$. The classic Weyl's dimension formula says
that
\begin{equation}\label{eqG.2}
\Dim_N\nu=\prod\limits_{1\le i<j\le N}\frac{\nu_i-\nu_j+j-i}{j-i}.
\end{equation}
Therefore, for $\ka\prec\nu$
\begin{equation}\label{eqF.17}
\La^N_{N-1}(\nu,\ka)=\frac{(N-1)!\prod\limits_{1\le i<j\le
N-1}(\ka_i-\ka_j+j-i)}{\prod\limits_{1\le i<j\le N}(\nu_i-\nu_j+j-i)}\,,
\end{equation}
otherwise $ \La^N_{N-1}(\nu,\ka)=0$.

Fix $\ka$ and assume $\nu$ is such that $\ka\prec\nu$. Then  $\nu\to\infty$ is
equivalent to either $\nu_1\to+\infty$, or $\nu_N\to -\infty$, or both; all
other coordinates of $\nu$ must remain bounded because of the interlacing
condition $\ka\prec\nu$. But then it is immediate that at least one of the
factors in the denominator of \eqref{eqF.17} tends to infinity. Thus, the ratio
goes to $0$ as needed.
\end{proof}

\subsection{Totality of $\{\varphi_\nu\}$}
Given $\nu\in\GT_N$, write the expansion of $S_\nu(u_1,\dots,u_N)$ in
monomials,
$$
S_\nu(u_1,\dots,u_N)=\sum c(\nu; n_1,\dots,n_N)u_1^{n_1}\dots,u_N^{n_N},
$$
where the sum is over $N$-tuples of integers $(n_1,\dots,n_N)$ with
$$
n_1+\dots+n_N=\nu_1+\dots+\nu_N.
$$
Obviously, the sum is actually finite. Further, the coefficients are
nonnegative integers: they are nothing else than the weight multiplicities of
the irreducible representation of $U(N)$ indexed by $\nu$. In purely
combinatorial way, this can be also deduced from the branching rule for the
characters: it follows that $c(\nu; n_1,\dots,n_N)$ equals the number of
triangular Gelfand--Tsetlin schemes $\{\la^{(j)}_i: 1\le i\le j\le N\}$ with
the top row $\la^{(N)}=\nu$ and such that
$$
\left(\la^{(j)}_1+\dots+\la^{(j)}_j\right)
-\left(\la^{(j-1)}_1+\dots+\la^{(j-1)}_{j-1}\right)=\nu_j, \qquad j=2,\dots,N.
$$

By virtue of Proposition
\ref{propF.4}, the functions $\varphi_n(\om)$ lie in $C_0(\Om)$. Therefore, the
same holds for the functions $\varphi_\nu(\om)$.

The results of the next proposition and its corollary are similar to
\cite[Lemme 3]{Vo76}, and the main idea of the proof is the same.

\begin{proposition}\label{propF.8}
For any $N=1,2,\dots$ and any $N$-tuple $(n_1,\dots,n_N)\in\Z^N$,
\begin{equation}\label{eqF.13}
\varphi_{n_1}(\om)\dots\varphi_{n_N}(\om)=\sum_\nu c(\nu;
n_1,\dots,n_N)\varphi_\nu(\om),
\end{equation}
where the series in the right-hand side converges in the norm topology of the
Banach space $C_0(\Om)$.
\end{proposition}

\begin{proof}
First, let us show that \eqref{eqF.13} holds pointwise. Indeed, this follows
from the comparison of the following two expansions:
\begin{align*}
\Phi(u_1;\om)\dots\Phi(u_n;\om)&=\sum_{(n_1,\dots,n_N)\in\Z^N}
\varphi_{n_1}(\om)\dots\varphi_{n_N}(\om) u_1^{n_1}\dots u_N^{n_N}
\\
&=\sum_{\nu\in\GT_N}\varphi_\nu(\om)S_\nu(u_1,\dots,u_N).
\end{align*}

Next, as all the functions in \eqref{eqF.13} are continuous and the summands in
the right-hand side are nonnegative, the series converges uniformly on compact
subsets of $\Om$.

Finally, as all the functions vanish at infinity, monotone convergence also
implies convergence in norm.
\end{proof}

\begin{corollary}\label{corF.2}
The family $\{\varphi_\nu: \nu\in\GT\}$ is total in the Banach space
$C_0(\Om)$, that is, the linear span of these functions is dense.
\end{corollary}

\begin{proof}
Let $\Om\cup\{\infty\}$ denote the one-point compactification of $\Om$. It
suffices to show that the family $\{\varphi_\nu: \nu\in\GT\}$ together with the
constant $1$ is total in the real Banach algebra $C(\Om\cup\{\infty\})$. By
Proposition \ref{propF.8}, the linear span of the family contains the
subalgebra generated by 1 and the functions $\varphi_n(\om)$. By virtue of
Proposition \ref{propF.3}, this subalgebra separates points of $\Om$. Next, for
any fixed $\om\in\Om$, the function $u\to\Phi(u;\om)$ cannot be identically
equal to 1, which implies that all the functions $\varphi_n$ cannot vanish at
$\om$ simultaneously. On the other hand, recall that they vanish at $\infty$.
This means that our subalgebra separates points of $\Om\cup\{\infty\}$, too.
Therefore, we may apply the Stone-Weierstrass theorem.
\end{proof}

\subsection{Description of the boundary}

\begin{theorem}\label{thmF.1}
For an arbitrary coherent system $\{M_K: K=1,2,\dots\}$ of distributions on the
graph\/ $\GT$ there exists a probability Borel measure $M$ on $\Om$ such that
\begin{equation}\label{eqF.14}
M_K=M\La^\infty_K, \qquad K=1,2,\dots,
\end{equation}
that is,
$$
M_K(\ka)=\int_\Om  M(d\om)\La^\infty_K(\om,\ka), \qquad \ka\in\GT_K, \quad
K=1,2,\dots\,,
$$
where $\La^\infty_K:\Om\dasharrow\GT_K$ is the Markov kernel defined in
subsection \ref{sectF.1}.

Such a measure is unique, and any probability Borel measure $M$ on\/ $\Om$
gives rise in this way to a coherent system.

\end{theorem}

In Section \ref{G} we reduce Theorem \ref{thmF.1} to Theorem \ref{thmG.1} whose
proof in turn is given in the subsequent sections.

Let us say that $M$ is the {\it boundary measure\/} of a given coherent system
$\{M_K\}$

By virtue of the theorem, the boundary measures of the extreme coherent systems
are exactly the delta measures on\/ $\Om$. Therefore, the theorem implies

\begin{corollary}\label{corF.1}
There exists a bijection $\pd(\GT)\leftrightarrow\Om$, under which the extreme
coherent system $\{M^{(\om)}_K: K=1,2,\dots\}$ corresponding to a point
$\om\in\Om$ is given by formula
$$
M_K^{(\om)}(\ka)=\La^\infty_K(\om,\ka), \qquad \ka\in\GT_K, \quad
K=1,2,\dots\,.
$$
\end{corollary}

Conversely, the theorem can be derived from the result of the corollary: the
necessary arguments can be found in \cite[Th\'eor\`eme 2]{Vo76} and
\cite[Theorems 9.1 and 9.2]{Ols03}.

\section{The Uniform Approximation Theorem}\label{G}

Recall the definition of the {\it modified Frobenius coordinates\/} of a Young
diagram $\la$ (see \cite{VK81}): First, introduce the conventional Frobenius
coordinates of $\la$:
$$
p_i=\la_i-i, \quad q_i=(\la')_i-i, \qquad i=1,\dots, d(\nu),
$$
where $\la'$ stands for the transposed diagram and $d(\la)$ denotes the number
of diagonal boxes of a Young diagram $\la$. The modified Frobenius coordinates
differ from the conventional ones by addition of one-halves:
$$
a_i=p_i+\tfrac12, \quad b_i=q_i+\tfrac12.
$$
Next, it is convenient to set
$$
a_i=b_i=0, \qquad i>d(\la),
$$
which makes it possible to assume that index $i$ ranges over $\{1,2,\dots\}$.
Note that $\sum_{i=1}^\infty (a_i+b_i)=|\la|$, where $|\la|$ denotes the total
number of boxes in $\la$.

Using the modified Frobenius coordinates we define for every $N=1,2,\dots$ an
embedding $\GT_N\hookrightarrow\Om$ in the following way. Let $\nu\in\GT_N$ be
given. We represent $\nu$ as a pair $(\nu^+,\nu^-)$ of partitions or,
equivalently, Young diagrams: $\nu^+$ consists of positive $\nu_i$'s, $\nu^-$
consists of minus negative $\nu_i$'s, and zeros can go in either of the two:
$$
\nu=(\nu_1^+,\nu_2^+,\dots,-\nu_2^-,\nu_1^-).
$$
Write $a^\pm_i, b^\pm_i$ for the modified Frobenius coordinates of $\nu^\pm$.
Then we assign to $\nu$ the point $\om(\nu)\in\Om$ with coordinates
$$
\al^\pm_i=\frac{a^\pm_i}N, \quad \be^\pm_i=\frac{b^\pm_i}N \quad (i=1,2,\dots),
\quad \de^\pm=\frac{|\nu^\pm|}N.
$$

Clearly, the correspondence $\GT_N\ni\nu\mapsto \om(\nu)$ is indeed an
embedding. The image of $\GT_N$ under this embedding is a locally finite set in
$\Om$: its intersection with any relatively compact subset is finite.

Note also that for points $\om=\om(\nu)$, $\de^\pm$ exactly equals
$\sum(\al^\pm_i+\be^\pm_i)$.

\begin{theorem}[Uniform Approximation Theorem]\label{thmG.1}
For any fixed $K=1,2,\dots$ and $\ka\in\GT_K$
\begin{equation}\label{eqG.5}
\lim_{N\to\infty}\sup_{\nu\in\GT_N}\left|\La^N_K(\nu,\ka)-\La^\infty(\om(\nu),\ka)\right|=0.
\end{equation}
\end{theorem}

\begin{proof}[Derivation of Theorem \ref{thmF.1} from Theorem \ref{thmG.1}]
We will verify the assertions of Theorem \ref{thmF.1} in the reverse order.

The fact that any probability Borel measure $M$ on $\Om$ serves as the boundary
measure of a coherent system $\{M_K: K=1,2,\dots\}$ is obvious from
\eqref{eqF.9}.

Next, if a coherent system $\{M_K\}$ has a boundary measure on $\Om$,  then its
uniqueness directly follows  from  Corollary \ref{corF.2}. Indeed, here we use
the fact that the space of finite signed measures on $\Om$ is dual to the
Banach space $C_0(\Om)$.

Now let us deduce from Theorem \ref{thmG.1} the existence of the boundary
measure for every coherent system $\{M_K\}$.

Write the compatibility relation for our coherent system,
$$
M_N\La^N_K=M_K, \qquad N>K,
$$
in the form
$$
\langle M_N, \, \La^N_K(\,\cdot\,,\ka)\rangle=M_K(\ka), \qquad N>K, \quad
\ka\in\GT_K,
$$
where $\La^N_K(\,\cdot\,,\ka)$ is viewed as the function $\nu\mapsto
\La^N_K(\nu,\ka)$ on $\GT_N$ and the angle brackets denote the canonical
pairing between measures and functions.

Denote by $\wt M_N$ the pushforward of $M_N$ under the embedding
$\GT_N\hookrightarrow\Om$ defined by $\nu\mapsto\om(\nu)$; this is a
probability measure on $\Om$ supported by the subset
$$
\wt{\GT_N}:=\{\om(\nu): \nu\in\GT_N\}\subset\Om.
$$

Next, regard $\La^{N}_K(\nu,\ka)$ as a function of variable $\om$ ranging over
$\wt{\GT_N}$ and denote this function by $\wt\La^N_K(\,\cdot\,,\ka)$.  Then we
may write the above compatibility relation as
$$
\langle \wt M_N, \, \wt\La^N_K(\,\cdot\,,\ka)\rangle=M_K(\ka), \qquad N>K.
\quad \ka\in\GT_K,
$$

By virtue of Theorem \ref{thmG.1}, for any $\om\in\wt{\GT_N}$
$$
\wt\La^N_K(\om,\ka)=\La^\infty_K(\om,\ka)+o(1), \qquad N\gg K,
$$
where the remainder term $o(1)$ depends on $\ka$ but is uniform on
$\om\in\wt{\GT_N}$. Since $\wt M_N$ is a probability measure, we get
\begin{equation}\label{eqG.7}
\langle \wt M_N, \, \La^\infty_K(\,\cdot\,,\ka)\rangle=M_K(\ka)+o(1), \qquad
N\gg K.
\end{equation}

The space $\Om$ is not only locally compact but also metrizable and separable.
Therefore, any sequence of probability measures on $\Om$ always has a nonvoid
set of partial limits in the vague topology (which is nothing else than the
weak-* topology of the Banach dual to $C_0(\Om)$). Note that, in general, it
may happen that such limits are sub-probability measures (the total mass is
strictly less than 1).

So, let $M$ be any partial vague limit of the sequence $\{\wt M_N\}$.
Passing to the limit in \eqref{eqG.7} along an appropriate subsequence of
indices $N$ we get
\begin{equation}\label{eqG.6}
\langle M, \, \La^\infty_K(\,\cdot\,,\ka)\rangle=M_K(\ka), \qquad K=1,2, \dots,
\quad \ka\in\GT_K,
\end{equation}
which is equivalent to the desired relation \eqref{eqF.14}.

Finally, once relation \eqref{eqF.14} is established, $M$ must be a probability
measure because otherwise the total mass of $M_K$ would be strictly less than
1, which is impossible.
\end{proof}

The rest of the section is a comment to Theorem \ref{thmG.1}, and the proof of
the theorem is given next in  Sections \ref{A}--\ref{E} .

\medskip

Recall that both $\La^N_K(\nu,\ka)$ and $\La^\infty_K(\om,\ka)$ involve one and
the same common factor $\Dim_K\ka$\,:
$$
\La^N_K(\nu,\ka)=\Dim_K\ka\cdot \frac{\Dim_{K,N}(\ka,\nu)}{\Dim_N\nu}, \qquad
\La^\infty_K(\om,\ka)=\Dim_K\ka\cdot\varphi_\nu(\om).
$$
As $\nu$ varies together with $N$, this factor remains intact. For this reason,
in what follows, we ignore it and study the relative dimension
\begin{equation}\label{eqG.1}
\frac{\Dim_{K,N}(\ka,\nu)}{\Dim_N\nu}.
\end{equation}
Incidentally, we get explicit formulas for this quantity (see Theorem
\ref{thmD.1} and its modification, Proposition \ref{propD.2}).

\begin{remark}\label{remG.1}
Recall that for the denominator in \eqref{eqG.1} there is a simple expression,
\eqref{eqG.2}. Let us also write down an expression for the numerator. Assume
that $\ka$ and $\nu$ are as in Theorem \ref{thmG.1}, and assume additionally
that $\ka_K\ge\nu_N$ (otherwise $\Dim_{K,N}(\ka,\nu)=0$). Define partitions
$\bar\nu$ and $\bar\ka$ as follows:
\begin{gather*}
\bar\nu=(\nu_1-\nu_N,\dots,\nu_{N-1}-\nu_N,0,0,\dots)\\
\bar\ka=(\ka_1-\nu_N,\dots,\ka_{K-1}-\nu_N,0,0,\dots).
\end{gather*}
We will also assume that $\bar\nu_i\ge\bar\ka_i$ for all $i=1,2,\dots$
(otherwise $\Dim_{K,N}(\ka,\nu)=0$). This enables us to define the skew Schur
function $S_{\bar\nu/\bar\ka}$. Then one has
\begin{equation}\label{eqG.3}
\Dim_{K,N}(\ka,\nu)=S_{\bar\nu/\bar\ka}(\,\underbrace{1,\dots,1}_{N-K}\,)
=\det\left[h_{\bar\nu_i-\bar\ka_j-i+j}(\,\underbrace{1,\dots,1}_{N-K}\,)\right]
\end{equation}
where the order of the determinant is any number greater than or equal to
$\ell(\bar\nu)$ (the number of nonzero coordinates in $\bar\nu$) and
\begin{equation}\label{eqG.4}
h_m(\,\underbrace{1,\dots,1}_{N-K}\,)=\begin{cases}\dfrac{(m+N-K-1)!}{m!(N-K-1)!},
& m\ge0\\ 0, & m<0 \end{cases}
\end{equation}
(The proof of the first equality in \eqref{eqG.3} is an easy exercise, and the
remaining equalities are standard facts.)

Combining \eqref{eqG.2}, \eqref{eqG.3}, and \eqref{eqG.4} we get a closed
expression for the relative dimension \eqref{eqG.1}. However, it is unclear how
one could use it for the problem of asymptotic analysis that we need. The formulas of
Section \ref{D}, on the contrary, are difficult to prove, but they have the
advantage to be well adapted to asymptotic analysis. Another their advantage is
that they involve determinants of order $K$, while the order of determinant in
\eqref{eqG.3} is generically $N-1$. Because of this, for $N\gg K$ and generic
$\nu$ the formulas of Section \ref{D} seem to be more efficient than
\eqref{eqG.3} from the purely computational viewpoint, too.

\end{remark}

\section{A Cauchy-type identity}\label{A}

The classical Cauchy identity for the Schur symmetric functions is
$$
\prod_{i,j}\frac1{1-x_iy_j}=\sum_\mu S_\mu(x_1,x_2,\dots)S_\mu(y_1,y_2,\dots),
$$
see e.g. \cite[Section I.4]{Ma95}. Here summation is over all partitions $\mu$
and  $S_\mu(x_1,x_2,\dots)$ denotes the Schur function indexed by $\mu$. For
finitely many indeterminates the identity takes the form
\begin{equation}\label{eqA.1}
\prod_{i=1}^N\prod_{j=1}^K\frac1{1-x_iy_j}=\sum_{\ell(\mu)\le\min(N,K)}
S_\mu(x_1,\dots,x_N)S_\mu(y_1,\dots,y_K).
\end{equation}
Here the Schur functions turn into the Schur polynomials and $\ell(\mu)$
denotes the length of partition $\mu$, i.e. the number of its nonzero parts.

The purpose of this section is to derive an analog of identity \eqref{eqA.1}
where the Schur polynomials in $x$'s are replaced by the shifted Schur
polynomials \cite{OO97a}, and the Schur polynomials in $y$'s are replaced by
other Schur-type functions, the dual symmetric Schur functions \cite{Mo09}. Let
us give their definition.

The {\it shifted Schur polynomial\/} with $N$ variables and index $\mu$ is
given by formula
\begin{equation}\nonumber
S^*_\mu(x_1,\dots,x_N)=\frac{\det[(x_i+N-i)^{\down
\mu_j+N-j}]}{\prod_{i<j}(x_i-x_j-i+j)},
\end{equation}
Here indices $i$ and $j$ range over $\{1,\dots,N\}$, and $x^{\down m}$ is our
notation for the $m$th falling factorial power of variable $x$,
\begin{equation}
x^{\down m}=\frac{\Ga(x+1)}{\Ga(x+1-m)}=x(x-1)\dots(x-m+1).
\end{equation}
The polynomial $S^*_\mu(x_1,\dots,x_N)$ is symmetric in shifted variables
$x'_i:=x_i-i$, and one has
\begin{equation}\nonumber
S^*_\mu(x_1,\dots,x_N)=S_\mu(x'_1,\dots,x'_N)+\textrm{lower degree terms}.
\end{equation}
This implies that, as functions in shifted variables $x'_1,\dots,x'_N$, the
polynomials $S^*_\mu$ form a basis in the ring $\C[x'_1,\dots,x'_N]^\sym$ of
$N$-variate symmetric polynomials. For more detail, see \cite{OO97a}.

By the {\it dual Schur symmetric function\/} in $K$ variables with index $\mu$
we mean the following function
\begin{equation}\label{eqA.6}
\si_\mu(t_1,\dots,t_K)=(-1)^{K(K-1)/2}\,
\dfrac{\det\left[\dfrac{\Ga(t_i+j-\mu_j)}{\Ga(t_i+1)}\right]}{\prod_{i<j}(t_i-t_j)},
\end{equation}
where $i$ and $j$ range over $\{1,\dots,K\}$ and the matrix in the numerator is
of order $K$. The $(i,j)$ entry of this matrix is a rational function in
variable $t_i$, so that $\si_\mu$ is a rational function in $t_1,\dots,t_K$.
Clearly, it is symmetric.

Let $\C(t_1,\dots,t_K)^\sym\subset\C(t_1,\dots,t_K)$ denote the subfield of
symmetric rational functions and
$\C(t_1,\dots,t_K)^\sym_\reg\subset\C(t_1,\dots,t_K)^\sym$ be the subspace of
functions regular about the point $(t_1,\dots,t_K)=(\infty,\dots,\infty)$. We
will also regard the space $\C(t_1,\dots,t_K)^\sym_\reg$ as a subspace in
$\C[[t_1^{-1},\dots,t_K^{-1}]]^\sym$, the ring of symmetric formal power series
in variables $t_1^{-1}, \dots,t_K^{-1}$. There is a canonical topology in this
ring: the $I$-adic topology determined by the ideal $I$ of the series without
the constant term. The Schur polynomials in $t_1^{-1}, \dots,t_K^{-1}$ form a
topological basis in $\C[[t_1^{-1},\dots,t_K^{-1}]]^\sym$, meaning that every
element of the ring is uniquely represented as an infinite series in these
polynomials.

We claim that functions $\si_\mu$ belong to $\C(t_1,\dots,t_K)^\sym_\reg$ and
form another topological basis in the ring
$\C[[t_1^{-1},\dots,t_K^{-1}]]^\sym$. Indeed, $\si_\mu$ is evidently symmetric.
To analyze its behavior about $(\infty,\dots,\infty)$, set $y_i:=t_i^{-1}$ and
observe that
\begin{equation}\nonumber
\frac{(-1)^{K(K-1)/2}}{\prod_{i<j}(t_i-t_j)}=\frac{(y_1\dots
y_K)^{K-1}}{\prod_{i<j}(y_i-y_j)}
\end{equation}
and
\begin{equation}\nonumber
y_i^{K-1}\frac{\Ga(t_i+j-\mu_j)}{\Ga(t_i+1)}=y_i^{\mu_j+K-j}+\textrm{higher
degree terms in $y_i$}.
\end{equation}
It follows that
$$
\si_\mu(t_1,\dots,t_K)=S_\mu(y_1,\dots,y_K)+\textrm{higher degree terms in
$y_1,\dots,y_K$},
$$
which entails our claim.

Note that functions $\si_\mu$ are a special case of more general {\it
multi-parameter\/} dual Schur functions defined in \cite{Mo09}.

In the definitions above we tacitly assumed that $\ell(\mu)$ does not exceed
the number of variables; otherwise the corresponding function is set to be
equal to zero. Under this convention the following {\it stability property\/}
holds:
\begin{equation}\nonumber
S^*_\mu(x_1,\dots,x_N)\big |_{x_N=0}=S^*_\mu(x_1,\dots,x_{N-1}), \quad
\si_\mu(t_1,\dots,t_K)\big |_{t_K=\infty}=\si_\mu(t_1,\dots,t_{K-1}).
\end{equation}
Both relations are verified in the same way as the stability property for the
ordinary Schur polynomials. The detailed argument for the first relation can be
found in \cite[Proposition 1.3]{OO97a}.

\begin{proposition}[Cauchy-type identity, cf. \eqref{eqA.1}]
One has
\begin{equation}\label{eqA.2}
\prod_{i=1}^N\prod_{j=1}^K\frac{t_j+i}{t_j+i-x_i}=\sum_{\ell(\mu)\le\min(N,K)}
S^*_\mu(x_1,\dots,x_N)\si_\mu(t_1,\dots,t_K).
\end{equation}
\end{proposition}

Here the infinite series in the right-hand side is the expansion with respect
to the topological basis $\{\si_\mu\}$ of
$(\C[x_1,\dots,x_N]^\sym)[[t_1^{-1},\dots,t_K^{-1}]]^\sym$, the topological
ring of symmetric formal power series in variables $t_1^{-1},\dots,t_K^{-1}$
with coefficient ring $\C[x_1,\dots,x_N]^\sym$. A more general form of the
identity can be found in \cite{Mo09}.

\begin{proof}
It suffices to prove \eqref{eqA.2} for $N=K$. Indeed, the general case is
immediately reduced to this one by making use of the stability property by
adding a few extra variables $x_i$ or $t_j$ and then specializing them to 0 or
$\infty$, respectively. Thus, in what follows we will assume $N=K$.

In the simplest case $N=K=1$, \eqref{eqA.2} takes the form
\begin{equation}\label{eqA.3}
\frac{t+1}{t+1-x}=\sum_{m=0}^\infty\frac{x^{\down m}}{t^{\down m}},
\end{equation}
which is just formula (12.3) in \cite{OO97a}.

Using \eqref{eqA.3} we will reduce the case $N=K>1$ of \eqref{eqA.2} to
Cauchy's determinant formula. Indeed, set $x'_i=x_i+K-i$, $m_i=\mu_i+K-i$, and
denote by symbol $V(\,\cdot\,)$ the Vandermonde in $K$ variables. Multiplying
the right-hand side of \eqref{eqA.2} by $V(x'_1,\dots,x'_K)V(t_1,\dots,t_K)$ we
transform it to
\begin{equation}\label{eqA.4}
(-1)^{K(K-1)/2}\prod_{j=1}^K\frac{\Ga(t_j+K)}{\Ga(t_j+1)}\,
\sum_{m_1>\dots>m_K\ge0}\det\left[{x'_i}^{\down
m_j}\right]\det\left[\dfrac1{(t_i+K-1)^{\down m_j}}\right],
\end{equation}
where both determinants are of order $K$.

A well-known trick allows one to write the sum in the right-hand side as a
single determinant:
$$
\sum_{m_1>\dots>m_K\ge0}\det\left[{x'_i}^{\down
m_j}\right]\det\left[\dfrac1{(t_i+K-1)^{\down m_j}}\right]=\det[A(i,j)]
$$
with
$$
A(i,j)=\sum_{m=0}^\infty\frac{{x'_i}^{\down m}}{(t_j+K-1)^{\down
m}}=\frac{t_j+K}{t_j+K-x'_i},
$$
where the last equality follows from \eqref{eqA.3}.

By Cauchy's determinant formula,
\begin{equation}\label{eqA.5}
\det[A(i,j)]=(-1)^{K(K-1)/2}\prod_{j=1}^K(t_j+K)\cdot
\frac{V(x'_1,\dots,x'_K)V(t_1,\dots,t_K)}{\prod_{i,j}(t_j+K-x'_i)}.
\end{equation}
Observe that $t_j+K-x'_i=t_j+i-x_i$. Taking this into account and plugging in
\eqref{eqA.5} instead of the sum in \eqref{eqA.4} we see that the plus-minus
sign disappears and the resulting expression for \eqref{eqA.4} coincides with
the left-hand side of \eqref{eqA.2} (for $N=K$) multiplied by the same product
of two Vandermonde determinants. This concludes the proof.
\end{proof}

\section{A generating function for the relative dimension}\label{B}

Throughout this section we assume that $N\ge K$ are two natural numbers, $\ka$
ranges over $\GT_K$ and $\nu$ ranges over $\GT_N$.

Set
\begin{equation}\label{eqB.7}
\begin{aligned} \Sf_{\ka\mid
N}(t_1,\dots,t_K)&=(-1)^{K(K-1)/2}\prod_{i=1}^K\frac{(N-K)!}{(N-K+i-1)!}\\
&\times \dfrac{\det\left[\dfrac{\Ga(t_i+1+N)\Ga(t_i+j-\ka_j)}
{\Ga(t_i+1)\Ga(t_i+j-\ka_j+N-K+1)}\right]} {V(t_1,\dots,t_K)},
\end{aligned}
\end{equation}
where the determinant is of order $K$ and
$V(t_1,\dots,t_K)=\prod_{i<j}(t_i-t_j)$, as above. The $(i,j)$ entry of the
matrix in the numerator is a rational function in $t_i$, which entails that
$\Sf_{\ka\mid N}(t_1,\dots,t_K)$ is an element of $\C(t_1,\dots,t_K)^\sym$.
Moreover, it is contained in $\C(t_1,\dots,t_K)^\sym_\reg$; this is readily
verified by passing to variables $y_i=t_i^{-1}$, as we already did in the case
of $\si_\mu$, see Section \ref{A}.

Next, in accordance with \cite[(12.3)]{OO97a}, we set
\begin{equation}\nonumber
H^*(t;\nu)=\prod_{j=1}^N\frac{t+j}{t+j-\nu_j}
\end{equation}
and more generally
\begin{equation}\nonumber
H^*(t_1,\dots,t_K;\nu)=H^*(t_1;\nu)\dots H^*(t_K;\nu).
\end{equation}
For $\nu$ fixed, $H^*(t_1,\dots,t_K;\nu)$ is obviously an element of
$\C(t_1,\dots,t_K)^\sym_\reg$, too.

Finally, recall the notation $\Dim_{K,N}(\ka,\nu)$ and $\Dim_N\nu$  introduced
in subsection \ref{sectF.2} We agree that $\Dim_{K,K}(\ka,\nu)$ is the
Kronecker delta $\de_{\ka\nu}$.

The purpose of this section is to prove the following claim.

\begin{proposition}\label{propB.1}
Let $N\ge K$. For any fixed $\nu\in\GT_N$, the function
$H^*(t_1,\dots,t_K;\nu)$ can be uniquely expanded into a finite linear
combination of the functions $\Sf_{\ka\mid N}(t_1,\dots,t_K)$, and this
expansion takes the form
\begin{equation}\label{eqB.4}
H^*(t_1,\dots,t_K;\nu)=\sum_{\ka\in\GT_K}\frac{\Dim_{K,N}(\ka,\nu)}{\Dim_N\nu}\,\Sf_{\ka\mid
N}(t_1,\dots,t_K).
\end{equation}
\end{proposition}

We regard this as a generating function for the quantities $\Dim(\ka,\nu)/\Dim
\nu$. In the case $K=1$, $\ka$ is simply an integer $k$, and the above
expansion turns into
\begin{equation}\nonumber
\begin{aligned}
H^*(t;\nu)&=\sum_{k\in\Z}\frac{\Dim_{K,N}(k,\nu)}{\Dim_N\nu}\,\frac{(t+1)\dots(t+N)}{(t+1-k)\dots(t+N-k)}\\
&=\sum_{k\in\Z}\frac{\Dim_{K,N}(k,\nu)}{\Dim_N\nu}\,H^*(t;(k^N)),
\end{aligned}
\end{equation}
where $(k^N)=(k,\dots,k)\in\GT_N$.

\begin{proof}
The proof is rather long and will be divided in a few steps. In what follows
$\mu$ always stands for an arbitrary partition with $\ell(\mu)\le K$.

{\it Step\/} 1. Set
\begin{equation}\label{eqB.12}
(N)_\mu=\prod_{i=1}^{\ell(\mu)}(N-i+1)_{\mu_i}=\prod_{i=1}^{\ell(\mu)}(N-i+1)\dots(N-i+\mu_i)
\end{equation}
and note that $(N)_\mu\ne0$ because $N\ge K\ge\ell(\mu)$.

Let
\begin{equation}\nonumber
D_{K,N}:\C[[t_1^{-1},\dots,t_K^{-1}]]^\sym\to
\C[[t_1^{-1},\dots,t_K^{-1}]]^\sym
\end{equation}
denote the linear operator defined on the topological basis $\{\si_\mu\}$ by
\begin{equation}\label{eqB.9}
D_{N,K}:\si_\mu\;\to\;\frac{(N)_\mu}{(K)_\mu}\,\si_\mu.
\end{equation}

We claim that
\begin{equation}\label{eqB.1}
H^*(t_1,\dots,t_K;\nu)=\sum_{\ka\in\GT_K}\La^N_K(\nu,\ka)D_{N,K}H^*(t_1,\dots,t_K;\ka).
\end{equation}
This is interpreted as an equality in $\C[[t_1^{-1},\dots,t_K^{-1}]]^\sym$.
Note that the sum is finite because for $\nu$ fixed, the quantity
$\La^N_K(\nu,\ka)$ does not vanish only for finitely many $\ka$'s.

Indeed, by virtue of \eqref{eqA.2}  we have
$$
H^*(t_1,\dots,t_k;\nu)=\sum_\mu
S^*_\mu(\nu_1,\dots,\nu_N)\si_\mu(t_1,\dots,t_K)
$$
and likewise
$$
H^*(t_1,\dots,t_k;\ka)=\sum_\mu
S^*_\mu(\ka_1,\dots,\ka_K)\si_\mu(t_1,\dots,t_K).
$$
Therefore, \eqref{eqB.1} is equivalent to
\begin{equation}\label{eqB.2}
\frac{S^*_\mu(\nu_1,\dots,\nu_N)}{(N)_\mu}=\sum_{\ka}\La^N_K(\nu,\ka)
\frac{S^*_\ka(\ka_1,\dots,\ka_K)}{(K)_\mu}.
\end{equation}
But \eqref{eqB.2} follows from the {\it coherence relation\/} for the shifted
Schur polynomials, which says that
\begin{equation}\label{eqB.3}
\frac{S^*_\mu(\nu_1,\dots,\nu_N)}{(N)_\mu}=\sum_{\la:\,
\la\prec\nu}\frac{\Dim_{N-1}\la}{\Dim_N\nu}\,\frac{S^*_\mu(\la_1,\dots,\la_{N-1})}{(N-1)_\mu}.
\end{equation}
See  \cite[(10.30)]{OO97a}, which coincides with \eqref{eqB.3} within an
obvious change of notation. To deduce \eqref{eqB.2} from \eqref{eqB.3} we use
induction on $N$. For the initial value $N=K$, \eqref{eqB.2} is trivial (with
the understanding that $\La^K_K$ is the identity matrix), and the induction
step is implemented by \eqref{eqB.3}, because $\La^N_K(\nu,\ka)$ satisfies the
same recursion
$$
\La^N_K(\nu,\ka)=\sum_{\la:\, \la\prec\nu}\frac{\Dim_{N-1}\la}{\Dim_N\nu}\;
\La^{N-1}_K(\la,\ka), \qquad N>K.
$$
This completes the proof of \eqref{eqB.1}.

{\it Step\/} 2.  Our next goal is to prove the equality
\begin{equation}\label{eqB.5}
\frac{\Sf_{\ka\mid N}(t_1,\dots,t_K)}{\Dim_K\ka}=D_{N,K}H^*(t_1,\dots,t_K;\ka).
\end{equation}
Then \eqref{eqB.4} will immediately follow from \eqref{eqB.1}. Note that
\eqref{eqB.5} does not involve $\nu$.

On this step we will check that \eqref{eqB.5} holds for $N=K$, that is
\begin{equation}\label{eqB.6}
\frac{\Sf_{\ka\mid K}(t_1,\dots,t_K)}{\Dim_K\ka}=H^*(t_1,\dots,t_K;\ka).
\end{equation}

By virtue of \eqref{eqB.7}, the left-hand side of \eqref{eqB.6} equals
\begin{equation}\nonumber
\frac{(-1)^{K(K-1)/2}}{\prod_{i=1}^K(K-1)!\cdot \Dim_K\ka\cdot
V(t_1,\dots,t_k)}\; \det\left[\dfrac{\Ga(t_i+1+K)\Ga(t_i+j-\ka_j)}
{\Ga(t_i+1)\Ga(t_i+j-\ka_j+1)}\right].
\end{equation}
Setting $k_j=\ka_j-j$, $j=1,\dots,K$, this expression can be easily transformed
to
\begin{equation}\nonumber
\frac{(-1)^{K(K-1)/2}\prod_{i,j=1}^K(t_i+j)}{V(k_1,\dots,k_K)
V(t_1,\dots,t_k)}\; \det\left[\dfrac1{t_i-k_j}\right].
\end{equation}
Since
$$
\det\left[\dfrac1{t_i-k_j}\right]=\frac{(-1)^{K(K-1)/2}V(k_1,\dots,k_K)
V(t_1,\dots,t_k)}{\prod_{i,j}(t_i-k_j)},
$$
the final result is
$$
\prod_{i,j=1}^K\frac{t_i+j}{t_i+j-\ka_j}=H^*(t_1,\dots,t_K;\ka),
$$
as desired.

{\it Step\/} 3. By virtue of Step 2, to prove \eqref{eqB.5} it suffices to show
that
$$
\frac{\Sf_{\ka\mid N}}{\Dim_K\ka}=D_{N,K}\left(\frac{\Sf_{\ka\mid
K}}{\Dim_K\ka}\right),
$$
or, equivalently,
\begin{equation}\label{eqB.8}
\Sf_{\ka\mid N}=D_{N,K}\Sf_{\ka\mid N}.
\end{equation}
A possible approach would consist in computing explicitly the expansion
$$
\Sf_{\ka\mid N}(t_1,\dots,t_K)=\sum_\mu C(\mu;N)\si_\mu(t_1,\dots,t_K)
$$
from which one could see that the coefficients satisfy the relation
$$
C(\mu;N)=\frac{(N)_\mu}{(K)_\mu}\, C(\mu;K).
$$
However, we did not work out this approach. Instead of it we adopt the
following strategy: {}From the definition of $D_{N,K}$, see \eqref{eqB.9}, it
is clear that it suffices to prove that
\begin{equation}\label{eqB.10}
\Sf_{\ka\mid N}=D_{N,N-1}\Sf_{\ka\mid N-1}, \quad \forall N>K.
\end{equation}
To do this we will show that $D_{N,N-1}$ can be implemented by a certain
difference operator in variables $(t_1,\dots,t_K)$. Then this will allow us to
easily verify \eqref{eqB.10}.

On this step we find the difference operator in question:
\begin{equation}\label{eqB.11}
D_{N,N-1}=\frac1{(N-1)^{\down K}}\,\frac1
V\circ\prod_{i=1}^K(t_i+N-(t_i+1)\tau_i))\circ V,
\end{equation}
where $V$ is the operator of multiplication by
$V(t_1,\dots,t_K)$, and $\tau_i$ is the shift operator
$$
(\tau f)(t):=f(t+1).
$$
acting on variable $t_i$.

To verify that \eqref{eqB.11} agrees with the initial definition of
$D_{N,N-1}$, see \eqref{eqB.9}, we have to prove that the difference operator
in the right-hand side of \eqref{eqB.11} acts on $\si_\mu$ as multiplication by
$(N)_\mu/(N-1)_\mu$.

By the very definition of $(N)_\mu$, see \eqref{eqB.12},
$$
\frac{(N)_\mu}{(N-1)_\mu}=\prod_{j=1}^K
\frac{N-j+\mu_j}{N-j}=\frac1{(N-1)^{\down K}}\, \prod_{j=1}^N (N-j+\mu_j).
$$
Taking into account the same factor $1/(N-1)^{\down K}$ in front of
\eqref{eqB.11} and the definition of $\si_\mu$ given in \eqref{eqA.6}, we see
that the desired claim reduces to the following one: the action of the
difference operator
$$
\prod_{i=1}^K(t_i+N-(t_i+1)\tau_i))
$$
on the function
$$
\det\left[\dfrac{\Ga(t_i+j-\mu_j)}{\Ga(t_i+1)}\right]
$$
amounts to multiplication by $\prod_{j=1}^N (N-j+\mu_j)$. This in turn reduces
to the following claim, which is easily verified:
$$
(t+N-(t+1)\tau)\; \dfrac{\Ga(t-m)}{\Ga(t+1)}\;=\;
(N+m)\,\dfrac{\Ga(t-m)}{\Ga(t+1)}, \quad \forall m\in\Z.
$$
This completes the proof of \eqref{eqB.11}.

{\it Step\/} 4. Here we will establish \eqref{eqB.10} with the difference
operator defined by \eqref{eqB.11}. By the definition of $\Sf_{\mu\mid N}$, see
\eqref{eqB.7}, we have to prove that operator
$$
\prod_{i=1}^K(t_i+N-(t_i+1)\tau_i))
$$
sends function
$$
\det\left[\dfrac{\Ga(t_i+N)\Ga(t_i+j-\ka_j)}
{\Ga(t_i+1)\Ga(t_i+j-\ka_j+N-K)}\right]
$$
to
$$
(N-K)^K\cdot \det\left[\dfrac{\Ga(t_i+1+N)\Ga(t_i+j-\ka_j)}
{\Ga(t_i+1)\Ga(t_i+j-\ka_j+N-K+1)}\right].
$$
This reduces to the following claim, which is easily verified: for any
$k\in\Z$,
$$
(t+N-(t+1)\tau))\dfrac{\Ga(t+N)\Ga(t-k)}
{\Ga(t+1)\Ga(t-k+N-K)}=(N-K)\dfrac{\Ga(t+1+N)\Ga(t-k)}
{\Ga(t+1)\Ga(t-k+N-K+1)}.
$$

Thus we have completed the proof of expansion \eqref{eqB.4}.

{\it Step\/} 5. It remains to prove the uniqueness claim of the proposition.
That is, to prove that the functions $\Sf_{\ka\mid N}(t_1,\dots,t_K)$ with $N$
fixed and parameter $\ka$ ranging over $\GT_K$ are linearly independent. It
suffices to do this for the minimal value $N=K$, because of relation
\eqref{eqB.8} and the fact that operator $D_{N,K}$ is invertible. Next, by
virtue of \eqref{eqB.6}, this is equivalent to the claim that the functions
$H^*(t_1,\dots,t_K;\ka)$ are linearly independent.

Recall that
$$
H^*(t_1,\dots,t_K;\ka)=H^*(t_1; \ka)\dots H^*(t_K;\ka),
$$
where
$$
H^*(t;\ka)=\prod_{j=1}^K\frac{t+j}{t+j-\ka_j}.
$$
The numerators of the fractions do not depend on $\ka$ and so may be ignored.
Set $k_j=\ka_j-j$ and observe that $k_1>\dots>k_K$. Thus, we are led to the
claim that the family of the functions
$$
f_{k_1,\dots,k_K}(t_1,\dots,t_K):=\prod_{i=1}^K\prod_{j=1}^K\frac 1{t_i-k_j}
$$
depending on an arbitrary $K$-tuple $k_1>\dots>k_K$ of integers is linearly
independent. But this is obvious, because for given a $K$-tuple of parameters,
the corresponding function $f_{k_1,\dots,k_K}(t_1,\dots,t_K)$ can be
characterized as the only function of the family that has a nonzero
multidimensional residue at $t_1=k_1, \dots, t_K=k_K$.

\end{proof}

The next proposition is used in informal Remark \ref{remB.1} below and then in
the proof of Proposition \ref{propE.1}.

\begin{proposition}\label{propB.2}
We have
\begin{equation}\label{eqB.14}
H^*(t;\nu)=\Phi(u;\om(\nu)),
\end{equation}
provided that variables $t$ and $u$ are related by the mutually inverse
linear-fractional transformations
\begin{equation}\label{eqB.13}
t=-\frac12+\frac{N}{u-1}, \qquad u=1+\frac{N}{t+\frac12}.
\end{equation}
\end{proposition}

\begin{proof}
Recall that $\om(\nu)$ is defined in terms of the modified Frobenius
coordinates $\{a_i^\pm, b_i^\pm: 1\le i\le d^\pm\}$ of the Young diagrams
$\nu^\pm$, see the beginning of Section \ref{G}. Set
$$
\wt\nu_i=\nu_i+\tfrac{N+1}2-i.
$$
That is,
$$
(\wt\nu_i,\dots,\wt\nu_N)=(\nu_1,\dots,\nu_N)+(\tfrac{N-1}2,\,
\tfrac{N-1}2-1\,,\dots,\, -\tfrac{N-1}2+1,\, -\tfrac{N-1}2)
$$

The next identity follows from \cite[Proposition 4.1]{BO05}  (cf.
\cite[Proposition 1.2]{IO03}):
\begin{equation}\label{eqB.15}
\prod_{i=1}^N\frac{s-\tfrac{N+1}2+i}{s-\wt\nu_i}=\prod_{i=1}^{d^+}\frac{s-\tfrac
N2+b_i^+}{s-\tfrac N2-a_i^+}\cdot \prod_{i=1}^{d^-}\frac{s+\tfrac
N2-b_i^-}{s+\tfrac N2+a_i^-}.
\end{equation}

Plug in $s=t+\frac{N+1}2$ into \eqref{eqB.15}, then the left-hand side equals
$H^*(t;\nu)$. Let us transform the right-hand side. Variables $s$ and $u$ are
related to each other via
$$
s=\frac N2\cdot\frac{u+1}{u-1}, \qquad u=\frac{s+\frac N2}{s-\frac N2}.
$$
Recall also that the coordinates of $\om(\nu)$ are given by
$$
\al^\pm_i=\frac{a^\pm_i}N, \quad \be^\pm_i=\frac{b^\pm_i}N, \quad
\de^\pm=\frac{|\nu^\pm|}N=\sum(\al^\pm_i+\be^\pm_i).
$$
{}From this it is easy to check that the right-hand side of \eqref{eqB.15}
equals
$$
\prod_{i=1}^{d^+}\frac{1+\be^+_i(u-1)}{1-\al^+_i(u-1)}\cdot
\prod_{i=1}^{d^-}\frac{1+\be^-_i(u^{-1}-1)}{1-\al^-_i(u^{-1}-1)}=\Phi(u;\om(\nu)),
$$
as desired.
\end{proof}

\begin{remark}\label{remB.1}
Let variables $t_1,\dots,t_K$ be related to variables $u_1, \dots,u_K$ via
\eqref{eqB.13}. Assume that variables $u_i$ are fixed and $N$ goes to infinity,
so that variables $t_i$ grow linearly in $N$. Then it is easy to check that in
this limit regime
$$
\Sf_{\ka\mid N}(t_1,\dots,t_K)\to S_\ka(u_1,\dots,u_K).
$$
Taking into account \eqref{eqB.14} we see that expansion \eqref{eqB.4} mimics
expansion \eqref{eqF.5}. (We recall that the latter expansion has the form
\begin{align*}
\Phi(u_1;\om)\dots\Phi(u_K;\om)& =\sum_{\ka\in\GT_K}\varphi_\ka(\om)
S_\ka(u_1,\dots,u_K) \\
& =\sum_{\ka\in\GT_K}\det[\varphi_{\ka_i-i+j}(\om)]_{i,j=1}^K
S_\ka(u_1,\dots,u_K).)
\end{align*}

This observation makes it plausible that if $N\to\infty$ and
$\nu\in\GT_N$ varies together with $N$ in such a way that $\om(\nu)$ converges
to a point $\om\in\Om$, then the relative dimension
$\Dim_{K,N}(\ka,\nu)/\Dim_N\nu$ tends to $\varphi_\ka(\om)$. However, the
rigorous proof of this assertion (and of the stronger one stated in the Uniform
Convergence Theorem) requires substantial efforts. The first step made in the
next section is to obtain a determinantal formula for the relative dimension
mimicking the determinantal formula
$$
\varphi_\ka(\om)=\det[\varphi_{\ka_i-i+j}(\om)]_{i,j=1}^K.
$$

\end{remark}

\section{A Jacobi-Trudi-type formula}\label{C}

The classical Jacobi-Trudi formula expresses the Schur function $S_\mu$ as a
determinant composed from the complete symmmetric functions $h_m$, which are
special cases of the Schur functions:
$$
S_\mu=\det[h_{\mu_i-i+j}].
$$
This formula can be obtained in various ways (see e.g. \cite[Ch. I,
(3.4)]{Ma95}, \cite[Section 7.16]{St99}). In particular, it can be easily
derived from the Cauchy identity  \eqref{eqA.1}: To do this one multiplies both
sides of identity  \eqref{eqA.1} by $V_K(y_1,\dots,y_K)$ and then
$S_\mu(x_1,\dots,x_N)$ is computed as the coefficient of the monomial
$y_1^{\mu_1+K-1}y_2^{\mu_2+K-2}\cdots y_K^{\mu_K}$ (cf. the second proof of
Theorem 7.16.1 in \cite{St99}). The same idea, albeit in a somewhat disguised
form, is applied in the proof of Proposition \ref{propC.1} below.

Observe that the structure of formula \eqref{eqB.7} for the functions
$\Sf_{\ka\mid N}$ is similar to that for the Schur polynomials. This suggests
the idea that identity \eqref{eqB.4} may be viewed as a kind of Cauchy
identity, so that one may expect a Jacobi-Trudi formula for the quantities
$\Dim_{K,N}(\ka,\nu)/\Dim_N\nu$. The purpose of the present section is to
derive such a formula. But first we have to introduce necessary notation.

For a finite interval $\LL$ of the lattice $\Z$, let $V_\LL$ denote the space
of rational functions in variable $t\in\C\cup\{\infty\}$, regular everywhere
including $t=\infty$, except possible simple poles at some points in
$\Z\setminus\LL$. Thus, $V_\LL$ is spanned by $1$ and the fractions
$(t-m)^{-1}$, where $m$ ranges over $\Z\setminus\LL$.

\begin{lemma}\label{lemmaC.1}
The functions
\begin{equation}\label{eqC.9}
f_{\LL,k}(t)=\frac{\prod_{x\in\LL}(t-x)}{\prod_{x\in\LL}(t-x-k)}, \quad k\in\Z,
\end{equation}
form one more basis in $V_\LL$.
\end{lemma}

\begin{proof}
Obviously, $f_{\LL,k}$ is in $V_\LL$ for every $k\in\Z$. In particular,
$f_{\LL,0}$ is the constant function $1$. On the other hand, given
$k=1,2,\dots$, any function in $V_{\LL}$ with the only possible poles on the
right of $\LL$, at distance at most $k$ from the right endpoint of $\LL$, can
be expressed through $f_{\LL,0}, \dots, f_{\LL,k}$, as is easily verified by
induction on $k$\,. Moreover, such an expression is unique. Likewise, the same
holds for functions with poles located on the left of $\LL$.
\end{proof}

By the lemma, any function $f\in V_\LL$ is uniquely written as a finite linear
combination
\begin{equation}\nonumber
f=\sum_{k\in\Z}c_k f_{\LL,k}.
\end{equation}
For the coefficients $c_k$ we will use the notation
\begin{equation}\nonumber
c_k=(f:f_{\LL,k}).
\end{equation}

Set
\begin{equation}\nonumber
\LL(N)=\{-N,\dots,-1\}.
\end{equation}
{}From the very definition of the function $H^*(t;\nu)$ one sees that it lies
in $V_{\LL(N)}$ for every $\nu\in\GT_N$. Consequently, the coefficients
$(H^*(\,\cdot\,;\nu):f_{\LL(N),k})$ are well defined. We also
need more general coefficients $(H^*(\,\cdot\,;\nu):f_{\LL,k})$, where $\LL$ is
a subinterval in $\LL(N)$. They are well defined, too, because
$V_{\LL}\supseteq V_{\LL(N)}$.

The coefficients $(H^*(\,\cdot\,;\nu):f_{\LL(N),k})$ will play the role of the
$h_k$-functions in variables $\nu=(\nu_1,\dots,\nu_N)$, while more general
coefficients $(H^*(\,\cdot\,;\nu):f_{\LL,k})$ should be interpreted as some
modification of those ``$h_k$-functions''. It is worth noting that the
conventional complete homogeneous symmetric functions are indexed by
nonnegative integers, while in our situation the index ranges over the set $\Z$
of all integers.

The purpose of the present section is to prove the following proposition.

\begin{proposition}[Jacobi-Trudi-type formula]\label{propC.1}
Let $N\ge K\ge1$, $\nu\in\GT_N$, and $\ka\in\GT_K$. For $j=1,\dots,K$, set
$$
\LL(N,j)=\{-N+K-j,\dots,-j\}.
$$
One has
\begin{equation}\label{eqC.1}
\frac{\Dim_{K,N}(\ka,\nu)}{\Dim_N\nu}
=\det\left[\left(H^*(\,\cdot\,;\nu):f_{\LL(N,j), \, \ka_i-i+j}\right)\right]_{i,j=1}^K,\\
\end{equation}
\end{proposition}

Note that the interval $\LL(N,j)\subset\Z$ comprises $N-K+1$ points and is
entirely contained in $\LL(N)$. As $j$ ranges from 1 to $K$, this interval
moves inside $\LL(N)$ from the rightmost possible position to the leftmost one.

In the simplest case $K=1$, \eqref{eqC.1} agrees with \eqref{eqB.4}. Indeed,
then the signature $\ka$ is reduced to a single integer $k\in\Z$, and formula
\eqref{eqC.1} turns into the following one
\begin{equation}\label{eqC.2}
\frac{\Dim_{1,N}(k,\nu)}{\Dim_N\nu}=\left(H^*(\,\cdot\,;\nu):f_{\LL(N),
k}\right).
\end{equation}
On the other hand, $f_{\LL(N), k}$ coincides with $\Sf_{k\mid N}$, so that
\eqref{eqC.2} is a special case of \eqref{eqB.4} corresponding to the
univariate case $K=1$:
\begin{equation}\nonumber
H^*(t;\nu)=\sum_{k\in\Z} \frac{\Dim_{1,N}(k,\nu)}{\Dim_N\nu} \Sf_{k\mid N}(t).
\end{equation}

A naive Jacobi-Trudi-type generalization of \eqref{eqC.2} to the case $K>1$
would consist in taking the determinant
$$
\det\left[\left(H^*(\,\cdot\,;\nu):f_{\LL(N), \, \ka_i-i+j}\right)\right].
$$
But this does not work, and it turns out that we have to appropriately modify
the univariate coefficients by shrinking $\LL(N)$ to a subinterval which varies
together with the column number $j$.

Note that a similar effect arises in the Jacobi-Trudi-type formula for the
shifted Schur functions or other variations of the Schur functions, see
\cite[Chapter I, Section 3, Example 21]{Ma95}, \cite[Section 13]{OO97a}: In the
Jacobi-Trudi determinant, the $h$-functions need to be appropriately modified
according to the column number.

\begin{proof}[Proof of Proposition \ref{propC.1}]
{\it Step\/} 1. Parameter $\nu$ being fixed, we will omit it from the notation
below. In particular, we abbreviate $H^*(t)=H^*(t;\nu)$. Assume we dispose of
an expansion into a finite sum, of the form
\begin{equation}\label{eqC.3}
H^*(t_1)\dots H^*(t_K)=\sum_{\ka}C(\ka) \Sf_{\ka\mid N}(t_1,\dots,t_K),
\end{equation}
with some coefficients $C(\ka)$. Then, due to the uniqueness claim of
Proposition \ref{propB.1}, the coefficients $C(\ka;\nu)$ must be the same as
the quantities $\Dim_{K,N}(\ka,\nu)/\Dim_N\nu$.

The functions $\Sf_{\ka\mid N}(t_1,\dots,t_K)$ can be written in the form
$$
\Sf_{\ka\mid
N}(t_1,\dots,t_K)=\const(N,K)\,\frac{\det\left[g_{k_j}(t_i)\right]_{i,j=1}^K}{V(t_1,\dots,t_K)},
$$
where
\begin{gather*}
k_1=\ka_1-1,\dots,k_K=\ka_K-K,\\
g_k(t)=\frac{\Ga(t+1+N)\Ga(t-k)}{\Ga(t+1)\Ga(t-k+N-K+1)}, \quad k\in\Z,\\
\const(N,K)=(-1)^{K(K-1)/2}\,\prod_{i=1}^K\frac{(N-K)!}{(N-K+i-1)!}.
\end{gather*}

Assume we have found some rational functions $\varphi_1(t),\dots, \varphi_K(t)$
with the following two properties:

$\bullet$ First, for every $a=1,\dots,K$ there exists a finite expansion
\begin{equation}\label{eqC.7}
H^*(t)\varphi_a(t)=\sum_{k\in\Z}C^a_k\,g_k(t)
\end{equation}
with some coefficients $C^a_k$.

$\bullet$ Second,
\begin{equation}\label{eqC.6}
\det\left[\varphi_a(t_i)\right]_{a,i=1}^K=\frac{V(t_1,\dots,t_K)}{\const(N,K)}.
\end{equation}

We claim that then \eqref{eqC.3} holds with coefficients
\begin{equation}\label{eqC.10}
C(\ka):=C(k_1,\dots,k_K):=\det\left[C^a_{k_b}\right]_{a,b=1}^K.
\end{equation}
Indeed, first of all, note that these coefficients vanish for all but finitely
many $\ka$'s (because of finiteness of expansion \eqref{eqC.7}), so that the
future expansion \eqref{eqC.3} will be finite. Next, applying \eqref{eqC.7} and
\eqref{eqC.6}, we have
\begin{gather*}
\sum_{k_1>\dots>k_K}C(k_1,\dots,k_K)\det\left[g_{k_j}(t_i)\right]_{i,j=1}^K
=\sum_{k_1>\dots>k_K}\det\left[C^a_{k_b}\right]_{a,b=1}^K
\det\left[g_{k_j}(t_i)\right]_{i,j=1}^K\\
=\det\left[\sum_{k\in\Z}C^a_k g_k(t_i)\right]_{i,a=1}^K
=\det\left[H^*(t_i)\varphi_a(t_i)\right]_{i,a=1}^K\\
=H^*(t_1)\dots H^*(t_K)\det\left[\varphi_a(t_i)\right]_{a,i=1}^K
=H^*(t_1)\dots H^*(t_K)\,\frac{V(t_1,\dots,t_K)}{\const(N,K)},
\end{gather*}
which is equivalent to the desired equality
\begin{equation*}
H^*(t_1)\dots H^*(t_K)=\sum_{\ka}\det\left[C^a_{k_b}\right]_{a,b=1}^K
\Sf_{\ka\mid N}(t_1,\dots,t_K).
\end{equation*}

{\it Step\/} 2. Now we exhibit the functions $\varphi_a(t)$:
\begin{equation}\label{eqC.11}
\varphi_a(t)=g_{-a}(t)=\frac{\Ga(t+a)\Ga(t+N+1)}{\Ga(t+1)\Ga(t+N-K+a+1)}, \quad
a=1,\dots,K.
\end{equation}
Let us examine what \eqref{eqC.7} means. Dividing the both sides of
\eqref{eqC.7} by $\varphi_a(t)$ we get
\begin{equation*}
H^*(t)=\sum_{k\in\Z}C^a_k\,\frac{g_k(t)}{\varphi_a(t)}.
\end{equation*}
But
$$
\frac{g_k(t)}{\varphi_a(t)}=\frac{\Ga(t+N-K+a+1)\Ga(t-k)}{\Ga(t+a)\Ga(t-k+N-K+1)}
=\frac{(t+a)(t+a+1)\dots(t+a+N-K)}{(t-k)(t-k+1)\dots(t-k+N-K)}.
$$
In the notation of \eqref{eqC.9}, this fraction is nothing else than
$f_{\LL,k+a}$, where $\LL$ denotes the interval $\{-N+K-a,\dots, -a\}$ in $\Z$.
It follows that the desired expansion \eqref{eqC.7} does exist and (restoring
the detailed notation $H^*(t;\nu)$) the corresponding coefficients are
$$
C^a_k=\left(H^*(\,\cdot\,;\nu): f_{\{-N+K-a,\dots, -a\},k+a}\right).
$$

Then the prescription \eqref{eqC.10} gives us
\begin{gather*}
C(\ka)=\det\left[C^a_{k_b}\right]_{a,b=1}^K=\det\left[C^a_{\ka_b-b}\right]_{a,b=1}^K\\
=\det\left[ \left(H^*(\,\cdot\,;\nu): f_{\{-N+K-a,\dots, -a\},\ka_b-b+a}\right)
\right]_{a,b=1}^K.
\end{gather*}
This is exactly \eqref{eqC.1}, within the renaming of indices $(b,a)\to(i,j)$.

{\it Step\/} 3. It remains to check that the functions \eqref{eqC.11} satisfy
\eqref{eqC.6}. That is, renaming $a$ by $j$,
\begin{multline*}
\det\left[\frac{\Ga(t_i+j)\Ga(t_i+N+1)}{\Ga(t_i+1)\Ga(t_i+N-K+j+1)}\right]_{i,j=1}^K\\
=V(t_1,\dots,t_K)\,(-1)^{K(K-1)/2}\prod_{i=1}^K\frac{(N-K+i-1)!}{(N-K)!}
\end{multline*}

or

\begin{multline}\label{eqC.12}
\det\left[(t_i+1)\dots(t_i+j-1)(t_i+N-K+j+1)\dots(t_i+N)\right]_{i,j=1}^K\\
=V(t_1,\dots,t_K)\,(-1)^{K(K-1)/2}\prod_{i=1}^K\frac{(N-K+i-1)!}{(N-K)!}.
\end{multline}

This identity is a particular case of Lemma 3 in Krattenthaler's paper
\cite{Kr99}. For the reader's convenience we reproduce the statement of this
lemma in the original notation:

Let $X_1,\dots,X_n$, $A_2,\dots,A_n$, and $B_2, \dots ,B_n$ be indeterminates.
Then
\begin{multline}\label{eqC.4}
\det\left[(X_i+A_n)(X_i+A_{n-1})\cdots(X_i+A_{j+1})
(X_i+B_j)(X_i+B_{j-1})\cdots (X_i+B_2)\right]_{i,j=1}^n\\
=\prod _{1\le i<j\le n} (X_i-X_j)\prod _{2\le i\le j\le n} ^{}(B_i-A_j).
\end{multline}

Setting $n=K$, $X_i=t_i$, $A_j=N-K+j$ and
$$
(B_2,\dots,B_n)=(1,\dots,K-1)
$$
one sees that the determinant in \eqref{eqC.4} turns into that in
\eqref{eqC.12}. Next,
$$
\prod _{1\le i<j\le n} ^{}(X_i-X_j)=V(t_1,\dots,t_K)
$$
and
\begin{gather*}
\prod _{2\le i\le j\le n} (B_i-A_j)=\prod _{1\le i< j\le K}
^{}(i-(N-K+j))\\
=(-1)^{K(K-1)/2}\prod _{1\le i< j\le K} (N-K+j-i))=(-1)^{K(K-1)/2}\prod_{j=1}^K
\frac{(N-K+j-1)!}{(N-K)!},
\end{gather*}
which agrees with \eqref{eqC.12}.

This completes the proof of the proposition.

\end{proof}

\section{Expansion on rational fractions}\label{D}

In this section we derive an expression for the coefficients
$(H^*(\,\cdot\,;\nu): f_{\LL,k})$ making formula \eqref{eqC.1} available for
practical use.

Fix a finite interval $\LL=\{a,a+1,\dots,b-1,b\}\subset\Z$. As explained in the
beginning of Section \ref{C}, the space $V_\LL$ has a basis consisting of the
rational fractions
$$
f_{\LL,k}(t)=\frac{(t-b)(t-b+1)\dots(t-a)}{(t-b-k)(t-b-k)\dots(t-a-k)}, \qquad
k\in\Z.
$$
For a rational function $G(t)$ from $V_\LL$, we write its expansion in the
basis $\{f_{\LL,k}\}_{k\in\Z}$ as
$$
G(t)=\sum_{k\in\Z}(G:f_{\LL,k}) f_{\LL,k}(t)
$$
and denote by $\Res_{t=x}G(t)$ the residue of $G(t)$ at a point $x\in\Z$.

\begin{proposition}\label{propD.1}
Assume $n:=b-a+1\ge2$. In the above notation
\begin{equation}\nonumber
(G:f_{\LL,k})=\begin{cases} (n-1)\sum\limits_{m\ge k}
\dfrac{(m-k+1)_{n-2}}{(m)_{n}}\,
\Res_{t=b+m}G(t), & k\ge1,\\
-(n-1)\sum\limits_{m\ge |k|} \dfrac{(m-|k|+1)_{n-2}}{(m)_{n}}\,
\Res_{t=a-m}G(t), & k\le-1,\\
G(\infty)+\sum\limits_{m\ge1}\dfrac1{m+n-1}\left(-\Res_{t=b+m}G(t)+\Res_{t=a-m}G(t)\right),
&k=0.
\end{cases}
\end{equation}
\end{proposition}

\begin{proof}
It is easy to write the expansion of $G(t)$ in another basis of $V_\LL$, formed
by $1$ and the fractions $(t-x)^{-1}$, where $x$ ranges over $\Z\setminus\LL$:
\begin{equation}\label{eqD.1}
G(t)=G(\infty)+\sum_{m\ge1}\left(\frac{\Res_{t=b+m}G(t)}{t-(b+m)}
+\frac{\Res_{t=a-m}G(t)}{t-(a-m)}\right).
\end{equation}
Thus, to find the coefficients $(G:f_{\LL,k})$ it suffices to compute the
expansion of the elements of the second basis on the fractions $f_{\LL,k}$.

Obviously,
\begin{equation}\label{eqD.4}
1=f_{\LL,0}.
\end{equation}
Thus, the problem is to expand the functions $(t-(b+m))^{-1}$ and
$(t-(a-m))^{-1}$ with $m=1,2,\dots$\,. We are going to prove that
\begin{gather}
\frac1{t-(b+m)}=-\frac1{m+n-1}f_{\LL,0}+\frac{n-1}{(m)_n}
\sum_{k=1}^m(m-k+1)_{n-2}f_{\LL,k} \label{eqD.2}\\
\frac1{t-(a-m)}=\frac1{m+n-1}f_{\LL,0}-\frac{n-1}{(m)_n}
\sum_{k=1}^m(m-k+1)_{n-2}f_{\LL,-k} \label{eqD.3}
\end{gather}
The claim of the proposition immediately follows from \eqref{eqD.1}--\eqref{eqD.3}.

Observe that \eqref{eqD.3} is reduced to \eqref{eqD.2} by making use of
reflection $t\to-t$. Indeed, under this reflection the basis formed by $1$ and
$\{f_{\LL,k}\}$ is transformed into the similar basis with $\LL$ replaced with
$-\LL$ (that is, parameters $a$ and $b$ are replaced by $-b$ and $-a$,
respectively), while the fractions from the second basis are transformed into
the similar fractions but multiplied by $-1$. This explains the change of sign
in the right-hand side of \eqref{eqD.3} as compared with \eqref{eqD.2}.

Thus, it suffices to prove identity \eqref{eqD.2}. Since it is invariant under
the simultaneous shift of $t$, $a$, and $b$ by an integer, we may assume, with
no loss of generality, that $a=1$, $b=n$. Then the identity takes the form
\begin{multline}\label{eqD.5}
\frac1{t-n-m}=-\frac1{m+n-1}\\ +\frac{n-1}{(m)_n} \sum_{k=1}^m
(m-k+1)_{n-2}\frac{(t-1)\dots(t-n)}{(t-1-k)\dots(t-n-k)}.
\end{multline}

The left-hand side vanishes at $t=\infty$. Let us check that the same holds for
the right-hand side. Indeed, this amounts to the identity
$$
\frac{n-1}{(m)_{n}}\sum_{k=1}^m (m-k+1)_{n-2}=\frac1{m+n-1}, \quad n\ge2.
$$
Renaming $n-1$ by $n$, the identity can be rewritten as
$$
n\sum_{k=1}^m (m-k+1)\cdots(m-k+n-1)=m\cdots(m+n-1), \quad n\ge1,
$$
and then it is easily proved by induction on $m$.

Next, the only singularity of the left-hand side of \eqref{eqD.5} is the simple
pole at $t=n+m$ with residue 1. Let us check that the right-hand side has the
same singularity at this point. Indeed, the only contribution comes from the
$m$th summand, which has a simple pole at $t=n+m$ with residue
$$
\frac{(n-1)!}{(m)_n}\,\frac{(t-1)\dots(t-n)}{(t-1-m)\dots(t-n-m+1)}\,\Bigg|_{t=n+m}=1,
$$
as desired.

It remains to check that the right-hand side of \eqref{eqD.5} is regular at
points $t=n+1, \dots, n+m-1$. All possible poles are simple, so that it
suffices to check that the residue at every such point vanishes. In the
corresponding identity, we may formally extend summation to $k=1,\dots, m+n-2$,
because the extra terms actually vanish. This happens due to the factor
$(m-k+1)_{n-2}$.

Thus, compute the residue at a given point $s\in\{n+1,\dots, n+m-1\}$. The
terms that contribute to the residue are those with $k=s-n, s-n+1,\dots, s-1$,
$n$ summands total. Setting $j=k-(s-n)$, the sum of the residues has the form
$$
\frac{n-1}{(m)_n}\sum_{j=0}^{n-1}(m-j-s+n+1)_{n-2}\,\frac{(-1)^j}{j!(n-1-j)!}\,
(s-1)\dots(s-n).
$$
The fact that this expression vanishes follows from a more general claim: For
any polynomial $P$ of degree $\le n-2$,
$$
\sum_{j=0}^{n-1}P(j) \,\frac{(-1)^j}{j!(n-1-j)!}=0.
$$

Finally, to prove the last identity, apply the differential operator
$(x\frac{d}{dx})^\ell$ to $(1-x)^{n-1}$ and then set $x=1$. For
$\ell=0,1,\dots,n-2$ this gives
$$
\sum_{j=0}^{n-1}j^\ell \,\frac{(-1)^j}{j!(n-1-j)!}=0.
$$

\end{proof}

Propositions \ref{propC.1} and \ref{propD.1} together give the following
explicit formula.

\begin{theorem}\label{thmD.1}
Let $\ka\in\GT_K$ and $\nu\in\GT_N$, where $N>K\ge1$, and recall the notation
$$
H^*(t;\nu)=\frac{(t+1)\dots(t+N)}{(t+1-\nu_1)\dots(t+N-\nu_N)}.
$$
One has
\begin{equation}\label{eqD.6}
\frac{\Dim_{K,N}(\ka,\nu)}{\Dim_N\nu}=\det\left[A_N(i,j)\right]_{i,j=1}^K,
\end{equation}
where the entries of the $K\times K$ matrix $A_N=\left[A_N(i,j)\right]$ are
defined according to the following rule, which depends on the column number
$j=1,\dots,K$ and the integer
$$
k:=k(i,j)=\ka_i-i+j.
$$
\medskip

$\bullet$ If $k\ge1$, then
\begin{equation}\label{f1'}
A_N(i,j)=(N-K)\sum\limits_{m\ge k} \dfrac{(m-k+1)_{N-K-1}}{(m)_{N-K+1}}\,
\Res_{t=-j+m}H^*(t;\nu)
\end{equation}

$\bullet$ If $k\le-1$, then
\begin{equation}\label{f2'}
A_N(i,j)= -(N-K)\sum\limits_{m\ge |k(i,j)|}
\dfrac{(m-|k|+1)_{N-K-1}}{(m)_{N-K+1}}\, \Res_{t=-N+K-j-m}H^*(t;\nu).
\end{equation}

$\bullet$ If $k=0$, then
\begin{equation}\label{f3'}
\begin{aligned}
A_N(i,j)=1&-\sum\limits_{m\ge 1}\dfrac1{m+N-K}\Res_{t=-j+m}H^*(t;\nu) \\
&+\sum\limits_{m\ge 1}\dfrac1{m+N-K} \Res_{t=-N+K-j-m}H^*(t;\nu).
\end{aligned}
\end{equation}

\end{theorem}

\begin{proof}
By virtue of Proposition \ref{propC.1},  formula \eqref{eqD.6} holds with the
$K\times K$ matrix $A_N=[A_N(i,j)]$ defined by
$$
A_N(i,j)=\left(H^*(\,\cdot\,;\nu):f_{\LL(N,j), \, \ka_i-i+j}\right), \quad
i,j=1,\dots,K,
$$
where
$$
\LL(N,j):=\{-N+K-j,\dots,-j\}\subseteq \LL(N):=\{-N,\dots,-1\}\subset\Z.
$$
To compute the entry $A_N(i,j)$ we apply Proposition \ref{propD.1}, where we
substitute $G(t)=H^*(t;\nu)$ and
$$
\LL=\LL(N,j), \quad a=-N+K-j, \quad b=-j, \quad n=N-K+1.
$$
This leads to the desired formulas.
\end{proof}

\begin{proposition}\label{propD.2}
Assume that $N$ is large enough, where the necessary lower bound depends on
$\ka$. Then the formulas of Theorem \ref{thmD.1} can be rewritten in the
following equivalent form{\rm:}

\medskip

$\bullet$ If $k\ge1$, then
\begin{equation}\label{f1}
A_N(i,j)=(N-K)\sum\limits_{\ell=0}^\infty
\dfrac{(\ell+j-k+1)_{k-1}}{(\ell+j-k+N-K)_{k+1}}\, \Res_{t=\ell} H^*(t;\nu).
\end{equation}

$\bullet$ If $k\le-1$, then
\begin{equation}\label{f2}
A_N(i,j)= -(N-K)\sum\limits_{\ell=-\infty}^{-N-1} \dfrac{(\ell+j
+N-K+1)_{|k|-1}}{(\ell+j)_{|k|+1}}\, \Res_{t=\ell} H^*(t;\nu).
\end{equation}

$\bullet$ If $k=0$, then
\begin{equation}\label{f3}
\begin{aligned}
A_N(i,j)=1&-\sum\limits_{\ell=0}^\infty\dfrac1{\ell+j+N-K}\Res_{t=\ell}H^*(t;\nu)\\
&-\sum\limits_{\ell=-\infty}^{-N-1}\dfrac1{-\ell-j} \Res_{t=\ell}H^*(t;\nu).
\end{aligned}
\end{equation}

\end{proposition}

\begin{proof}
Examine formula \eqref{f1'}. Its transformation to \eqref{f1} involves three
steps.

{\it Step\/} 1. The key observation is that the summation in \eqref{f1'} can be
formally extended by starting it from $m=1$. The reason is that the extra terms
with $1\le m<k$ actually vanish. Indeed, the vanishing comes from the product
$$
(m-k+1)_{N-K-1}=(m-k+1)\dots(m-k+N-K-1).
$$
Since $1\le m<k$, the first factor of the product is $\le0$ while the last
factor is positive (here the assumption that $N$ is large enough is
essential!). Therefore, one of the factors is 0.

{\it Step\/} 2. A simple transformation shows that
$$
\frac{(m-k+1)_{N-K-1}}{(m)_{N-K+1}}=\frac{\Gamma(m-k+N-K)\Gamma(m)}{\Gamma(m-k+1)\Gamma(m+N-K+1)}=\frac{(m-k+1)_{k-1}}{(m-k+N-K)_{k+1}}.
$$

{\it Step\/} 3. Observe that the possible poles of $H^*(t;\nu)$ are located in
$$
\Z\setminus\LL=\{\dots, -N-3, -N-2,-N-1\}\cup\{0,1,2,\dots\}.
$$
All possible poles at points $t=-j+m$, where $m=1,2,\dots$, are entirely
contained in $\{0,1,2,\dots\}$. Therefore, we may assume that $m$ ranges over
$\{j,j+1,j+2,\dots\}$. Setting $m=j+\ell$ we finally arrive at \eqref{f1}.

\medskip

To transform \eqref{f2'} to \eqref{f2} we apply the similar argument.

To transform the sums in \eqref{f3'} we need to apply only the last step of the
above argument.
\end{proof}

\section{Contour integral representation}\label{E}

We keep the notation of the preceding section: The number $K=1,2,\dots$ and
the signature $\ka\in\GT_K$ are fixed, and we are dealing with the $K\times K$
matrix $[A_N(i,j)]$ that depends on $\ka\in\GT_K$ and $\nu\in\GT_N$, and is
defined by the formulas of Proposition \ref{propD.2}. We denote by $\T$ the
unit circle $|u|=1$ in $\C$ oriented counterclockwise.

\begin{proposition}\label{propE.1}
Every entry $A_N(i,j)$ can be written in the form
\begin{equation}\label{eqE.1}
A_N(i,j)=\frac1{2\pi i}\oint\limits_{\T}\Phi(u;\om(\nu))R_{\ka_i-i+j}^{(j)}
(u;N)\frac{du}u,
\end{equation}
where, for any $k\in\Z$, $j=1,\dots,K$, and natural $N>K$, the function $u\to
R_k^{(j)}(u;N)$ is continuous on\/ $\T$ and such that
\begin{equation}\label{eqE.2}
\lim_{N\to\infty} R_k^{(j)}(u;N)=\frac1{u^k}
\end{equation}
uniformly on $u\in\T$.

The explicit expression for $R_k^{(j)}(u;N)$ is the following\,{\rm:}

$\bullet$ If $k\ge1$, then
\begin{equation}\label{eqE.5}
R_k^{(j)}(u;N)=\frac{N-K}{N}\,\dfrac{u\prod\limits_{m=1}^{k-1}
\left(1+\dfrac{(j-k-\frac12+m)(u-1)}N\right)}
{\prod\limits_{m=1}^{k+1}\left(u+\dfrac{(j-k-K-\frac32+m)(u-1)}N\right)}.
\end{equation}

$\bullet$ If $k\le-1$, then
\begin{equation}\label{eqE.10}
R_k^{(j)}(u;N)=\frac{N-K}{N}\,\dfrac{u\prod\limits_{m=1}^{|k|-1}
\left(u+\dfrac{(j-K-\frac12+m)(u-1)}N\right)}
{\prod\limits_{m=1}^{|k|+1}\left(1+\dfrac{(j-\frac32+m)(u-1)}N\right)}.
\end{equation}

$\bullet$ If $k=0$, then
\begin{equation}\label{eqE.11}
R_0^{(j)}(u;N)=\frac{N-K}{N}\,\dfrac{u}
{\left(u+\dfrac{(j-K-\frac12)(u-1)}N\right)\left(1+\dfrac{(j-\frac12)(u-1)}N\right)}.
\end{equation}
\end{proposition}

\begin{proof}
Recall (see Proposition \ref{propB.2}) that
$$
H^*(t;\nu)=\Phi(u;\om(\nu)),
$$
where $t$ and $u$ are related by the mutually inverse linear-fractional
transformations
$$
t=-\frac12+\frac{N}{u-1}, \qquad u=1+\frac{N}{t+\frac12}.
$$
The transformation $t\to u$ maps the right half-plane $\Re\, t>-\frac{N+1}2$
onto the exterior of the unit circle $|u|=1$, and the left half-plane $\Re\,
t<-\frac{N+1}2$ is mapped onto the interior of the circle. The vertical line
$\Re\,t=-\frac{N+1}2$ just passes through the midpoint of the interval
$[-N,-1]$, which is free of the poles of $H^*(t;\nu)$. Note also that
\begin{equation}\label{eqE.3}
dt=-\frac N{(u-1)^2}du=-\frac{Nu}{(u-1)^2}\,\frac{du}u.
\end{equation}

Consider separately the three cases depending on whether $k\ge1$, $k\le -1$ or
$k=0$.

{\it Case\/}  $k\ge1$. We can write \eqref{f1} as the contour integral
$$
A(i,j)=\frac1{2\pi i}\oint\limits_C \rho(t)H^*(t;\nu)dt,
$$
where $C$ is a simple contour in the half-plane $\Re t>-\frac{N+1}2$, oriented
in the positive direction and encircling all the poles of $H^*(t;\nu)$ located
in this half-plane and
$$
\rho(t)=(N-K)\frac{(t+j-k+1)_{k-1}}{(t+j-k+N-K)_{k+1}}.
$$
Passing to variable $u$ we get, after a simple transformation,
$$
\rho(t)=\wt\rho(u):=\frac{(N-K)(u-1)^2}{N^2}\, \dfrac{\prod\limits_{m=1}^{k-1}
\left(1+\dfrac{(j-k-\frac12+m)(u-1)}N\right)}
{\prod\limits_{m=1}^{k+1}\left(u+\dfrac{(j-k-K-\frac32+m)(u-1)}N\right)}.
$$

Without loss of generality we can assume that contour $C$ also encircles the
special point $t=-\frac12$ corresponding to $u=\infty$. This means that its
image in the $u$-plane goes around the unit circle $|u|=1$ in the negative
direction. Thus, we can deform it, in the $u$-plane, to the unit circle. The
change of orientation to the positive one produces the minus sign, which
cancels the minus sign in formula \eqref{eqE.3} for the transformation of the
differential. Note that the deformation of the contour is justified, because
$\wt\rho(u)$ has no singularity in the exterior of the unit circle  (this is
best  seen from the expression for $\rho(t)$, which obviously has no
singularity in the half-plane $\Re\,t>-\frac{N+1}2$). As for the factor
$(u-1)^2$ in the denominator of \eqref{eqE.3}, it is cancelled by the same
factor in the numerator of $\wt\rho(u)$. Finally we get the desired integral
representation \eqref{eqE.1} with $R^{(j)}_k(u;N)$ given by \eqref{eqE.5}.

{\it Case\/}  $k\le-1$. This case is analyzed in the same way.

{\it Case\/}  $k=0$. The same argument as above allows one to write the
expression in \eqref{f3} as
\begin{equation}\label{eqE.4}
1-\frac1{2\pi i} \oint\limits_{C_+}
\frac{\Phi(u;\om(\nu))du}{(u-1)(u+\epsi_1(u-1))} +\frac1{2\pi i}
\oint\limits_{C_-} \frac{\Phi(u;\om(\nu))du}{(u-1)(1+\epsi_2(u-1))}
\end{equation}
where
$$
\epsi_1=\frac{j-K-\tfrac12}N, \qquad \epsi_2=\frac{j-\tfrac12}N,
$$
and $C_+$ and $C_-$ are two circles close to the unit circle $|u|=1$, both
oriented in the positive direction, and such that $C_+$ is outside the unit
circle while $C_-$ is inside it. Since $\Phi(u;\om(\nu))$ takes value 1 at
$u=1$, we have
$$
1-\frac1{2\pi i} \oint\limits_{C_+}
\frac{\Phi(u;\om(\nu))du}{(u-1)(1+\epsi_2(u-1))} +\frac1{2\pi i}
\oint\limits_{C_-} \frac{\Phi(u;\om(\nu))du}{(u-1)(1+\epsi_2(u-1))}=0.
$$
Subtracting this from \eqref{eqE.4} we get the contour integral with the
integrand equal to $\Phi(u;\om(\nu))du$ multiplied by
\begin{gather*}
\frac1{(u-1)(1+\epsi_2(u-1))}-\frac1{(u-1)(u+\epsi_1(u-1))}\\
=\frac{1+\epsi_1-\epsi_2}{(u+\epsi_1(u-1))(1+\epsi_2(u-1))}.
\end{gather*}
This leads to \eqref{eqE.11}.

The asymptotics \eqref{eqE.2} is obvious from the explicit expressions
\eqref{eqE.5}, \eqref{eqE.10}, and \eqref{eqE.11}.
\end{proof}

Let $\T^K=\T\times\dots\times\T$ be the $K$-fold product of unit circles.
Theorem \ref{thmD.1}, Proposition \ref{propD.2}, and Proposition \ref{propE.1}
together imply the following result.

\begin{theorem}\label{thmE.1}
Given $K=1,2,\dots$ and $\ka\in\GT_K$, one can exhibit a sequence
$\{R_\ka(u_1,\dots,u_K;N): N>K\}$ of continuous functions on the torus $\T^K$
such that\/{\rm:}

{\rm(i)} For all $N$ large enough and every $\nu\in\GT_N$
\begin{multline}\label{eqE.7}
\frac{\Dim_{K,N}(\ka,\nu)}{\Dim_N\nu} =\frac1{(2\pi
i)^K}\oint_\T\dots\oint_\T\Phi(u_1;\om(\nu))\dots\Phi(u_K;\om(\nu))\\
\times R_\ka(u_1,\dots,u_K;N)\frac{du_1}{u_1}\dots \frac{du_K}{u_K},
\end{multline}
where each copy of\/ $\T$ is oriented counterclockwise.

{\rm(ii)} As $N$ goes to infinity,
$$
R_\ka(u_1,\dots,u_K;N)\to\det\left[u_j^{-(\ka_i-i+j)}\right]_{i,j=1}^K
$$
uniformly on $(u_1,\dots,u_K)\in\T^K$.
\end{theorem}

\begin{proof}
(i) Indeed, set
\begin{equation}\label{eqE.6}
R_\ka(u_1,\dots,u_K;N)=\det\left[R_{\ka_i-i+j}^{(j)}(u_j;N)\right]_{i,j=1}^K,
\end{equation}
where the functions $R^{(j)}_{\ka_i-i+j}(u,N)$ are defined in Proposition
\ref{propE.1}. Recall that Theorem \ref{thmD.1} expresses the relative
dimension $\Dim_{K,N}(\ka,\nu)/\Dim_N\nu$ as the determinant of a matrix
$A_N=[A_N(i,j)]$; Proposition \ref{propD.2} provides a more convenient
expression for the matrix entries that works for large $N$; and finally
Proposition \ref{propE.1} says that this expression can be written as a contour
integral involving the functions $R^{(j)}_{\ka_i-i+j}(u,N)$. Now we plug in the
determinant \eqref{eqE.6} into the $K$-fold contour integral \eqref{eqE.7} and
expand the determinant on columns. Applying \eqref{eqE.1} we get
$\det[A_N(i,j)]$, as desired.

(ii) This follows directly from \eqref{eqE.6} and \eqref{eqE.2}.
\end{proof}

\begin{remark}\label{remE.1}
The graph $\GT$ possesses the {\it reflection symmetry\/} $\nu\mapsto\wh\nu$,
where, given a signature $\nu\in\GT_N$, $N=1,2,\dots$, we set
$$
\wh\nu=(\wh\nu_1,\dots,\wh\nu_N):=(-\nu_N,\dots,-\nu_1).
$$
The corresponding symmetry $\om\mapsto\wh\om$ of $\Om$ amounts to switching the
plus- and minus-coordinates:
$$
\al^+_i\leftrightarrow\al^-_i, \qquad \be^+_i\leftrightarrow\be^-_i, \qquad
\de^+\leftrightarrow\de^-.
$$
Note also that $\wh{\om(\nu)}=\om(\wh\nu)$ and
$$
\Phi(u;\om)=\Phi(u^{-1};\wh\om).
$$

Evidently, the reflection symmetry preserves the relative dimension:
$$
\frac{\Dim_{K,N}(\ka,\nu)}{\Dim_N\nu}=\frac{\Dim_{K,N}(\wh\ka,\wh\nu)}{\Dim_N\wh\nu}.
$$
Therefore, the expression given in \eqref{eqE.7} must satisfy this identity.
This is indeed true and can be readily verified using the relation
$$
R^{(j)}_k(u;N)=R^{(K+1-j)}_{-k}(u^{-1};N),
$$
which follows directly from \eqref{eqE.5}, \eqref{eqE.10}, and \eqref{eqE.11}.
\end{remark}

\smallskip

The Uniform Approximation Theorem (Theorem \ref{thmG.1}) is a direct
consequence of Theorem \ref{thmE.1}:

\begin{proof}[Proof of the Uniform Approximation Theorem]
As was already pointed out in the end of Section \ref{G}, both quantities
$\La^N_K(\nu,\ka)$ and $\La^\infty_K(\om,\ka)$ entering \eqref{eqG.5} involve
one and the same constant factor $\Dim_K\ka$. Therefore, \eqref{eqG.5} is
equivalent to
\begin{equation}\label{eqE.8}
\lim_{N\to\infty}\sup_{\nu\in\GT_N}\left|\frac{\Dim_{K,N}(\ka,\nu)}
{\Dim_N\nu}-\varphi_\ka(\om(\nu))\right|=0.
\end{equation}

To estimate the deviation
\begin{equation}\label{eqE.9}
\frac{\Dim_{K,N}(\ka,\nu)} {\Dim_N\nu}-\varphi_\ka(\om(\nu))
\end{equation}
we observe that both quantities in \eqref{eqE.9} can be written as $K$-fold
contour integrals of the same type. Indeed, for the relative dimension we apply
\eqref{eqE.7}. Next, by the very definition
$\varphi_\ka(\om)=\det[\varphi_{\ka_i-i+j}(\om)]$ and
$$
\varphi_k(\om)=\frac1{2\pi i}\oint_\T\Phi(u;\om)\frac1{u^k}\,\frac{du}u,
$$
so that $\varphi_\ka(\om(\nu))$ admits a similar integral representation, only
$R(u_1,\dots,u_K;N)$ has to be replaced by
$$
\det\left[u_j^{-(\ka_i-i+j)}\right]_{i,j=1}^K.
$$

It follows that for any $\nu\in\GT_N$ the modulus of \eqref{eqE.9} is bounded
from above by the following integral over the torus $\T^K$ taken with respect
to the normalized Lebesgue measure $m(du)$, where we abbreviate
$u=(u_1,\dots,u_K)$:
$$
\int_{\T^K}|\Phi(u_1;\om(\nu))\cdots\Phi(u_K;\om(\nu))|
\left|R_\ka(u_1,\dots,u_K;N)-\det\left[u_j^{-(\ka_i-i+j)}\right]_{i,j=1}^K\right|
m(du).
$$

By Proposition \ref{propF.6},
$$
|\Phi(u_1;\om(\nu))\cdots\Phi(u_K;\om(\nu))|\le1.
$$
Therefore, the above integral does not exceed
$$
\int_{\T^K}\left|R_\ka(u_1,\dots,u_K;N)-\det\left[u_j^{-(\ka_i-i+j)}\right]_{i,j=1}^K\right|
m(du)
$$
and the desired uniform bound follows from the second assertion of Theorem
\ref{thmE.1}.
\end{proof}

\section{Appendix}\label{I}

Let $\{\nu(N)\in\GT_N:N=1,2,\dots\}$ be a sequence of signatures of growing
length. We say that it is {\it regular\/} if for any fixed $K=1,2,\dots$ the
sequence of probability measures $\La^N_K(\nu(N),\,\cdot\,)$ weakly converges
to a probability measure on $\GT_K$. (This means that for every $\ka\in\GT_K$
there exists a limit $\lim_{N\to\infty}\La^N(\nu(N),\ka)$ and the sum over
$\ka\in\GT_K$ of the limit values equals $1$.) This definition is equivalent to
regularity of the sequence of normalized characters $\wt\chi_{\nu(N)}$ as
defined in \cite{OO98}.

A particular case of the results of \cite{OO98} is the following theorem.

\begin{theorem}\label{thmI.1}
A sequence $\{\nu(N)\in\GT_N:N=1,2,\dots\}$ is regular if and only if the
corresponding sequence $\{\om(\nu(N))$ of points in $\Om$ converges to a point
$\om\in\Om$.

Moreover, if $\{\nu(N)\in\GT_N:N=1,2,\dots\}$ is regular, then the the limit
measure $\lim_{N\to\infty}\La^N_K(\nu(N),\,\cdot\,)$ coincides with
$\La^\infty_K(\om,\,\cdot\,)$, where $\om=\lim_{N\to\infty}\om(\nu(N))$.
\end{theorem}

The aim of this section is to discuss the interrelations between this assertion
and the Uniform Convergence Theorem (Theorem \ref{thmG.1}). Recall that this
theorem says that for any fixed $\ka\in\GT_K$
\begin{equation}\label{eqI.1}
\lim_{N\to\infty}\sup_{\nu\in\GT_N}\left|\La^N_K(\nu,\ka)-\La^\infty(\om(\nu),\ka)\right|=0
\end{equation}

\begin{proof}[Derivation of Theorem \ref{thmI.1} from Theorem \ref{thmG.1}]
Combining \eqref{eqI.1}  with continuity of $\La^\infty_K(\om,\ka)$ in the
first argument we see that if the sequence $\om(\nu(N))$ converges to a point
$\om\in\Om$, then for any fixed $K$, the measure $\La^N_K(\nu(N),\,\cdot\,)$
weakly converges to the probability measure $\La^\infty_K(\om,\,\cdot\,)$, so
that $\{\nu(N)\}$ is regular.

Conversely, assume that $\{\nu(N)\}$ is regular and prove that
$\{\om(\nu(N))\}$ has a limit $\om\in\Om$. Since $\Om$ is locally compact, it
suffices to prove that $\{\om(\nu(N))\}$ cannot have two distinct limit points
in $\Om$ and cannot contain a subsequence converging to infinity.

The existence of distinct limit points is excluded by virtue of the argument
above and the fact that different points of $\Om$ generate different measures
on $\GT_1$, which in turn follows from Proposition \ref{propF.3}.

The escape to infinity for a subsequence is also impossible, as is seen from
\eqref{eqI.1} and the fact that $\La^\infty_K(\om,\ka)\to0$ as $\om$ goes to
infinity.

This completes the proof.
\end{proof}

\begin{proof}[Derivation of Theorem \ref{thmG.1} from Theorem \ref{thmI.1}
and results from \cite{Ok97}] It suffices to prove the following assertion: If
$N(1)<N(2)<\dots$ and $\nu(1)\in\GT_{N(1)},\nu(2)\in\GT_{N(2)}, \dots$ are such
that for any fixed $K$ and $\ka\in\GT_K$ there exists a limit
$$
\lim_{n\to\infty}\left(\La^{N(n)}_K(\nu(n),\ka)-\La^\infty(\om(\nu(n)),\ka)\right)=c_\ka,
$$
then $c_\ka=0$ for all $\ka$.

Passing to a subsequence we are led to the following two cases: either the
sequence $\{\om(\nu(N))\}$ converges to a point $\om\in\Om$ or this sequence
goes to infinity.

In the first case, the desired assertion follows from Theorem \ref{thmI.1}.
Indeed, it says that $\La^{N(n)}_K(\nu(n),\ka)\to\La^\infty_K(\om,\ka)$. On the
other hand, $\La^\infty_K(\om(\nu(n)),\ka)\to\La^\infty_K(\om,\ka)$ by
continuity of $\La^\infty_K(\om,\ka)$.

In the second case, we know that $\La^\infty_K(\om(\nu(n)),\ka)\to0$ for any
$\ka$ (see Corollary \ref{corF.3}). Therefore, we have to prove that for any
$K$, the measures $M^{(n)}_K:=\La^{N(n)}_K(\nu(n),\,\cdot\,)$ weakly converge
to 0.

Without loss of generality we may assume that for every $K$ the sequence
$\{M^{(n)}_K\}$ weakly converges to a measure $M^{(\infty)}_K$. It follows
(here we also use the  Feller property of the stochastic matrices
$\La^{K+1}_K$, see Proposition \ref{propF.9}) that the limit measures are
compatible with these matrices:
$$
M^{(\infty)}_{K+1}\La^{K+1}_K=M^{(\infty)}_K, \qquad K=1,2,3,\dots\,.
$$
Therefore, the total mass of $M^{(\infty)}_K$ does not depend on $K$. If this
mass equals 1, that is, the limit measures are probability measures, then
Theorem \ref{thmI.1} implies that the sequence $\om(\nu(N(n))$ converges in
$\Om$, which is impossible. If the total mass equal 0, the limit measures are
zero measures and we are done. Thus, it remains to prove that the total mass of
$M^{\infty}_K$ cannot be equal to a number strictly contained between  0 and 1.
It suffices to prove this assertion for $K=1$, and then it is the subject of
the proposition below, which relies on results of \cite{Ok97}.
\end{proof}

\begin{proposition}
Let $M^{(1)},M^{(2)}, \dots$ be a sequence of probability measures on $\Z$ such
that every $M^{(n)}$ has the form $\La^N_1(\nu,\,\cdot\,)$, where $N\ge2$ and
$\nu\in\GT_N$ depend on $n$. Then  $\{M^{(n)}\}$ cannot weakly converge to a
nonzero measure of total mass strictly less than\/ $1$.
\end{proposition}

In other words, such a sequence of probability measures cannot escape to
infinity {\it partially\/}.

\begin{proof}
A measure $M$ on $\Z$ is said to be {\it log-concave\/} if for any two integers
$k,l$ of the same parity
$$
M(k)M(l)\le(M(\tfrac12(k+l)))^2.
$$

Each measure of the form $M=\La^N_1(\nu,\,\cdot\,)$ is log-concave: this
nontrivial fact is a particular case of the results of \cite{Ok97}.

Furthermore, such a measure has no {\it internal zeros\/}, that is, its support
is a whole interval in $\Z$. Indeed, it is not hard to check that the support
of $\La^N_1(\nu,\,\cdot\,)$ is the interval $\{\nu_N,\dots,\nu_1\}\subset\Z$.

Thus, our probability measures $M^{(n)}$ are log-concave and have no internal
zeros. Assume that they weakly converge to a {\it nonzero\/} measure
$M^{(\infty)}$. Then we may apply the argument of \cite[p. 276]{Ok97}. It
provides a uniform on $n$ bound on the tails of measures $M^{(n)}$, which shows
that for any $r=1,2,\dots$, the $r$th moment of $M^{(n)}$ converges to the
$r$th moment of $M^{(\infty)}$. The convergence of the second moments already
suffices (via Chebyshev's inequality) to conclude that $M^{(\infty)}$ is a
probability measure.
\end{proof}

\end{document}